\documentclass[11pt]{amsart}
\usepackage{amssymb}
\usepackage{mathrsfs}
\usepackage{cite}
\usepackage{amscd}
\usepackage{stmaryrd}

{\theoremstyle{definition} 
\newtheorem{dfn}{Definition}
\newtheorem{ex}{Example}
\newtheorem{q}{Question}}
\newtheorem{pro}{Proposition}
\newtheorem{lmm}{Lemma}
\newtheorem{thm}{Theorem}
\newtheorem{cor}{Corollary}

\newcommand{\HHi}{{\mathcal{H}{\rm ilb}}}
\newcommand{\Hi}{{{\rm Hilb}}}

\begin{document} 

\title{Gr\"obner strata in the Hilbert scheme of points}
\author{Mathias Lederer}
\email{mlederer@math.cornell.edu}
\address{Department of Mathematics, Cornell University, Ithaca, New York 14853, USA}
\date{January, 2011}

\begin{abstract} 
  The present paper shall provide a framework for working with Gr\"obner bases over arbitrary rings $k$
  with a prescribed finite standard set $\Delta$. 
  We show that the functor associating to a $k$-algebra $B$ the set of all reduced Gr\"obner bases 
  with standard set $\Delta$ is representable
  and that the representing scheme is a locally closed stratum in the Hilbert scheme of points.
  We cover the Hilbert scheme of points by open affine subschemes which represent the functor 
  associating to a $k$-algebra $B$ the set of all border bases with standard set $\Delta$
  and give reasonably small sets of equations defining these schemes. 
  We show that the schemes parametrizing Gr\"obner bases are connected;
  give a connectedness criterion for the schemes parametrizing border bases;
  and prove that the decomposition of the Hilbert scheme of points into the locally closed strata 
  parametrizing Gr\"obner bases is not a stratification.
\end{abstract}

\maketitle

%%%%%%%%%%%%%%%%%%%%%%%%%%%%%%%%%%%%%%

\section{Introduction}

Let $k$ be a ring and $S=k[x_{1},\ldots,x_{n}]$ be the polynomial ring over $k$, 
equipped with a term order $\prec$. 
There are various notions of Gr\"obner bases, and of reduced Gr\"obner bases, 
of an ideal $I\subset S$ (see \cite{paueroverview} for an overview). 
We use that notion of a reduced Gr\"obner basis which is 
entirely analogous to the definition in the case where $k$ is a field. 
The definition will be given in Section \ref{notation}. 
The same notion of a reduced Gr\"obner basis is used in \cite{wibmer}, 
a paper which was a significant source of inspiration for the work presented here.
However, not every ideal $I$ has a reduced Gr\"obner basis in this sense;
a reduced Gr\"obner basis exists if, and only if, $I$ is a {\it monic} ideal. 
Attached to a monic ideal is its standard set, 
which is the set of those elements of $\mathbb{N}^n$ which do not occur as the multidegree of an element of $I$. 
For getting a feeling for monic ideals, let us look at some examples. 

\begin{ex}
  Let $n=1$ and $I=(f)\subset k[x]$ be a principal ideal. 
  Then $I$ is a monic ideal if, and only if, it is generated by a monic polynomial. 
  In this case the reduced Gr\"obner basis of $I$ consists of $f$ alone. 
  The standard set of $I$ is $\{0,\ldots,\deg(f)-1\}\subset\mathbb{N}$. 
\end{ex}

\begin{ex}
  Let $n=2$. We equip $S$ with the lexicographic order such that $x_1\succ x_2$. 
  The ideal $I_1\subset \mathbb{Z}[x_1,x_2]$ generated by
  \begin{equation*}
    \begin{split}
      f&=x_2^3-3x_2^2+2x_2\,,\\
      g&=x_1^2x_2-x_1^2+x_1x_2-x_1+x_2-1\,,\\
      h&=x_1^4+2x_1^3x_2^2-4x_1^3x_2+x_1+x_1^2x_2^2-x_1^2x_2-x_1^2\\
      &-50x_1x_2^2+49x_1x_2+x_1+50x_2^2-49x_2-1
    \end{split}
  \end{equation*}
  is monic with standard set 
  \begin{equation*}
    \Delta=\{(0,0),(1,0),(2,0),(3,0),(0,1),(1,1),(0,2),(1,2)\}\subset\mathbb{N}^2\,. 
  \end{equation*}
  The reduced Gr\"obner basis of $I_1$ consists of $f$, $g$ and $\widetilde{h}$, where
  \begin{equation*}
    \begin{split}
      \widetilde{h}&=x_1^4+x_1^3-x_1^2+8x_1x_2^4-46x_1x_2^3+35x_1x_2^2-2x_1x_2+x_1\\
      &+8x_2^4-42x_2^3+20x_2^2+13x_2-1\,.
    \end{split}
  \end{equation*}
\end{ex}

\begin{ex}
  Using the notation of the previous example, the ideal 
  \begin{equation*}
    I_2=(2f,g,h)\subset \mathbb{Z}[x_1,x_2] 
  \end{equation*}
  is not monic, and accordingly, does not have a reduced Gr\"obner basis. 
  The reason for that is the exponent $(0,3)\in\mathbb{N}^2$, which appears as the multidegree of a monic element of $I_1$, 
  but not as the multidegree of any monic element of $I_2$. 
\end{ex}

Before outlining our article, let us briefly summarize its main ideas:
\begin{itemize}
  \item A good notion of Gr\"obner basis over an arbitrary ring is that of the reduced Gr\"obner basis of a monic ideal. 
  \item The reducedness property guarantees functoriality of Gr\"obner bases. 
  \item Functoriality guarantees the existence of a moduli space. 
  \item Some familiar techniques for Gr\"obner bases over fields can be carried over to the setup over rings. 
  \item $S$-pair criteria as such are not needed. 
\end{itemize}

If $B$ is a $k$-algebra and $\Delta$ is a finite standard set, 
we attach to $B$ the set of all monic ideals $I\subset B[x]$ with standard set $\Delta$. 
As the reduced Gr\"obner basis of a monic ideal is unique, 
we may equivalently attach to $B$ the set of all reduced Gr\"obner bases in $B[x]$ with standard set $\Delta$. 
It turns out that this map is functorial in $B$. 
We denote the functor by $\HHi^{\prec\Delta}_{S/k}$. 
Note the dependence of this functor on both $\Delta$ and the term order $\prec$. 
The notation is motivated by the fact that $\HHi^{\prec\Delta}_{S/k}$ 
is a subfunctor of the {\it Hilbert functor of points} $\HHi^{d}_{S/k}$. 
The Hilbert functor of points has been widely studied 
(see \cite{iarrobino}, \cite{huibregtse}, \cite{norge}, \cite{bertin} and references therein).
In particular, it is well-known that this functor is represented by a scheme $\Hi^{d}_{S/k}$.
The notions of Hilbert functor and Hilbert scheme were introduced by Grothendieck in \cite{grothendieck}; 
see \cite{nitsure} for an introductory account of the subject.
In our paper we will show that $\HHi^{\prec\Delta}_{S/k}$ is a locally closed subfunctor of $\HHi^{d}_{S/k}$, 
hence representable by a locally closed subscheme of the Hilbert scheme. 
We will study an intermediate functor $\HHi^{\Delta}_{S/k}$, which is also representable, 
such that in the chain of representing objects
\begin{equation}\label{chain}
  \Hi^{\prec\Delta}_{S/k}\subset\Hi^{\Delta}_{S/k}\subset\Hi^{d}_{S/k}\,,
\end{equation}
the first inclusion is a closed immersion and the second inclusion is an open immersion. 

The moduli spaces $\Hi^{\Delta}_{S/k}$ and $\Hi^{\prec\Delta}_{S/k}$ have been studied by numerous authors, 
at least in the case where $k$ is a field. 
In the article \cite{krarticle}, the scheme $\Hi^{\Delta}_{S/k}$ is called {\it border basis scheme}, 
and in the article \cite{robbiano}, the scheme $\Hi^{\prec\Delta}_{S/k}$
is called {\it Gr\"obner basis scheme}. 
In the cited papers, and in \cite{kk1}, \cite{kk2}, \cite{kkr}, \cite{krbook},
a theory of border bases, which generalizes the theory of Gr\"obner basis, is developed. 
In the present paper we use border bases as well, in studying the functor $\Hi^{\Delta}_{S/k}$. 
Some of the results of the cited papers are parallel to those of our article here. 
Each time we state one such result, we will indicate its relation to the cited papers. 
However, the two major differences between the cited papers and our treatment here are that firstly
our treatment is more general, as our $k$ is an arbitrary ring, 
and secondly we use the functorial language. 

In fact, the starting point for the work presented here was the desire to obtain a {\it relative}, 
i.e. {\it functorial} notion of Gr\"obner bases. 
The desire for functoriality is what motivates the use of arbitrary rings rather than fields. 
However, functoriality only holds if the leading terms of a Gr\"obner basis are stable under arbitrary tensor products. 
Therefore we have to use monic ideals, and correspondingly, reduced Gr\"obner bases. 
Moreover, several of our results are novel, or stronger than previous results, 
even if we specialize to the case where $k$ is a field. 

The paper which bears the closest relationship to our paper here is \cite{huibregtse}.
Huibregtse studies the functors $\HHi^{\Delta}_{S/k}$ and schemes $\Hi^{\Delta}_{S/k}$, as we do here. 
(A standard set $\Delta$ in our notation corresponds to a {\it basis set} $\beta$ in his, 
and $\Hi^{\Delta}_{S/k}$ in our notation is $U_\beta$ in his.)
In this sense Huibretse also covers the functorial properties of $\Hi^{\Delta}_{S/k}$. 
However, his viewpoint is different from ours:
In Lemma 7 of \cite{huibregtse}, he shows how to glue the schemes $\Hi^{\Delta}_{S/k}$, 
for all standard sets $\Delta\subset\mathbb{N}^n$ of size $d$, to obtain the Hilbert scheme $\Hi^d_{S/k}$. 
This implies in particular that each $\HHi^{\Delta}_{S/k}$ is an open subfunctor of $\HHi^{d}_{S/k}$, 
and that the various functors $\HHi^{\Delta}_{S/k}$ form an open cover of $\HHi^{d}_{S/k}$. 
These two statements are Lemma \ref{opensubfunctor} and Proposition \ref{cover}, resp., in the present paper. 
In contrast to Huibregtse's approach, 
we show these statements directly at the level of functors, without using the representing schemes. 
Moreover, Huibretse's construction of $\Hi^{\Delta}_{S/k}$ is entirely different from ours. 
His is based on {\it pseudosyzygies} and {\it syzygies}, 
whereas ours avoids the use of $S$-pair criteria altogether, as was mentioned above. 
Furthermore, the notion of reduced Gr\"obner bases and monic ideals over a ring $k$ do not appear in \cite{huibregtse}. 
Accordingly, the schemes $\Hi^{\prec\Delta}_{S/k}$, which we are mostly interested in here, 
do not appear in the cited paper. 
In this sense our results on the functorial properties of $\Hi^{\prec\Delta}_{S/k}$ are original. 

Another line of work is to study strata analogous to ours in Grothendieck's classical Hilbert scheme 
${\rm Hilb}^{p(z)}_{\mathbb{P}^n_{k}}$, where $p(z)$ is a polynomial. 
The paper \cite{notari} is devoted to this project;
the equations defining the strata are derived from Buchberger's $S$-pair criterion. 
However, the cited paper contains a few inconsistencies, as is indicated in \cite{robbiano} and \cite{lella}. 
In particular, in the latter paper, 
the embedding of the strata in ${\rm Hilb}^{p(z)}_{\mathbb{P}^n_{k}}$ is elaborated upon with care. 
Also, it appears to be the first paper in which the term {\it Gr\"obner stratum} is used. 
Other papers in which related ideas appear are \cite{carra} and \cite{roggero}. 

As was mentioned above, the research presented here was largely inspired by the paper \cite{wibmer}. 
In that paper, $k$ is a noetherian ring. Wibmer considers an arbitrary ideal $I\subset S$. 
The canonical map $k\to S/I$ corresponds to a morphism of affine schemes $\phi:{\rm Spec}\,S/I\to{\rm Spec}\,k$. 
The main theorem of \cite{wibmer} (Theorem 11) states the existence of a unique decomposition of ${\rm Spec}\,k$ 
into a finite number of locally closed strata such that on each stratum the reduction of $I$ to each point of the stratum
has a reduced Gr\"obner basis of a prescribed shape. 
It is striking to note the analogy of that theorem to Theorem \ref{coprod} of our paper here. 
However, in Wibmer's setting $k$ has to be noetherian, 
whereas our setting requires no restriction on $k$. 

Our article is organized as follows. 
In Section \ref{notation}, we introduce the basic notions of monic ideals, reduced Gr\"obner bases and standard sets. 
In Section \ref{subfunctors}, we define the Hilbert functor $\HHi^{d}_{S/k}$ and the open subfunctors $\HHi^{\Delta}_{S/k}$, 
where $\Delta$ runs through all standard sets of size $d$. 
In Section \ref{standardcovering}, we thoroughly prove that these subfunctors cover the whole functor $\HHi^{d}_{S/k}$. 
That gives us the key to defining the subfunctor $\HHi^{\prec\Delta}_{S/k}$ of $\HHi^{\Delta}_{S/k}$
in Section \ref{bases}. 
In Section \ref{groebnersubfunctors}, we show that from representability of $\HHi^{\Delta}_{S/k}$, 
representability of $\HHi^{\prec\Delta}_{S/k}$ follows. 
In Section \ref{decomposition}, we show that the Hilbert scheme $\Hi^{d}_{S/k}$ 
is the disjoint union of the representing schemes $\Hi^{\prec\Delta}_{S/k}$. 
Thus far the techniques we use are non-explicit in the sense that we use abstract representability criteria for functors 
rather than explicit descriptions of representing schemes. 
Once functoriality is proved, we turn to more concrete questions. 
In Section \ref{affines}, we write down a set of equations defining the affine schemes 
$\Hi^{\Delta}_{S/k}$ and $\Hi^{\prec\Delta}_{S/k}$. 
In Section \ref{examples}, we study a few examples and improve the result of the previous section
in shrinking the set of equations defining the affine schemes. 
In Section \ref{universal}, we write down the universal objects of the functors $\HHi^{\Delta}_{S/k}$ and 
$\HHi^{\prec\Delta}_{S/k}$, 
which are affine schemes over $\Hi^{\Delta}_{S/k}$ and $\Hi^{\prec\Delta}_{S/k}$, resp. 
In Section \ref{homogeneous}, we show that $\HHi^{\prec\Delta}_{S/k}$ is connected, 
we present a homogeneous variant of what we have done so far, and we give a connectedness criterion for $\HHi^{\Delta}_{S/k}$. 
In Section \ref{changingcharts}, we explore the transition maps between $\Hi^{\Delta}_{S/k}$ and $\Hi^{\Pi}_{S/k}$, 
we track $\Hi^{\prec\Delta}_{S/k}$ in $\Hi^{\Pi}_{S/k}$, 
and we show that the decomposition of Theorem \ref{coprod} in general is not a stratification. 

%%%%%%%%%%%%%%%%%%%%%%%%%%%%%%%%%%%%%%

\section{Notation}\label{notation}

We start by collecting the relevant definitions and facts concerning elements 
and ideals in $S$. 
Throughout, a monomial order $\prec$ on $S$ will be fixed. 
This is, in particular, a total order on the set of monomials $x^\alpha=x_{1}^{\alpha_{1}}\ldots x_{n}^{\alpha_{n}}$, 
where $\alpha=(\alpha_{1},\ldots,\alpha_{n})\in\mathbb{N}^n$. 
Let $f\in S$, then the monomial order gives the well-known definitions of
\begin{itemize}
  \item coefficient ${\rm coef}(f,x^\alpha)$ of $f$ at $x^\alpha$;
  \item support ${\rm supp}(f)$, which is the set of all $x^\alpha$ such that ${\rm coef}(f,x^\alpha)\neq0$;
  \item leading monomial ${\rm LM}(f)$;
  \item leading coefficient ${\rm LC}(f)$;
  \item leading term ${\rm LT}(f)$, which equals ${\rm LC}(f){\rm LM}(f)$;
  \item leading exponent (or multidegree) ${\rm LE}(f)$, for which ${\rm LM}(f)=x^{{\rm LE}(f)}$;
  \item the non-leading exponents, which are those $\alpha$ such that $x^\alpha$ lies in 
  ${\rm supp}(f)$ but does not equal ${\rm LM}(f)$.
\end{itemize}

If $I\subset S$ is an ideal, we let ${\rm LM}(I)$ be the set of all ${\rm LM}(f)$, where $f$ runs through $I-\{0\}$. 
This set is closed with respect to multiplication by arbitrary monomials. 
Analogously, we let ${\rm LT}(I)$ be the set of all ${\rm LT}(f)$, where $f$ runs through $I-\{0\}$. 
This set is also closed with respect to multiplication by arbitrary monomials. 
Clearly if $k$ is a field, then ${\rm LT}(I)$ carries the same information as ${\rm LM}(I)$ does, 
but if $k$ is a ring, then in general ${\rm LT}(I)$ carries more information than ${\rm LM}(I)$ does. 

If $I$ is an ideal in $S$ and $x^\alpha$ is a monomial, the set
\begin{equation*}
  {\rm LC}(I,x^\alpha)=\{{\rm LC}(f);f\in I-\{0\},{\rm LM}(f)=x^\alpha\}\cup\{0\}
\end{equation*}
is an ideal in $k$. 
\begin{dfn}
  An ideal $I$ is called {\it monic} (see \cite{pauer}, Definition 3.3 or \cite{wibmer}, Definition 4)
  if the following equivalent conditions are satisfied:
  \begin{itemize} 
    \item For all monomials $x^\alpha$, the ideal ${\rm LC}(I,x^\alpha)$ is either the zero ideal or the unit ideal;
    \item each element of ${\rm LM}(I)$ arises as the leading monomial of a monic $f\in I$;
    \item ${\rm LT}(I)$ is a monomial ideal;
    \item the sets ${\rm LM}(I)$ and ${\rm LT}(I)$ carry the same information.
  \end{itemize}
\end{dfn}

Note that if $k$ is a field, then each ideal in $S$ is generated by finitely many monic polynomials, hence in particular monic. 
Also note that if $k$ is an arbitrary ring and $I$ is a monic ideal, then $I$ is finitely generated. 

We will mostly be working with leading exponents, more precisely, 
with the set ${\rm LE}(I)$, which is the set of all ${\rm LE}(f)$, where $f$ runs through $I$. 
Clearly ${\rm LE}(I)$ carries the same information as ${\rm LM}(I)$. 
Therefore it carries the same information as ${\rm LT}(I)$ if, and only if, $I$ is monic. 
In fact, we will not be working with ${\rm LE}(I)$ itself but rather with its complement in $\mathbb{N}^n$:

\begin{dfn}
  \begin{itemize}
    \item A {\it standard set} (or {\it staircase}, or {\it Gr\"obner escalier}) in $\mathbb{N}^n$ 
    is a subset $\Delta\subset\mathbb{N}^n$ such that
    its complement in $\mathbb{N}^n$ is closed with respect to addition with elements of $\mathbb{N}^n$. 
    (Equivalently, standard sets are precisely the complements of the sets ${\rm LE}(I)$, 
    where $I$ runs through all ideals in $S$.)
    \item If for a given $I$ we have ${\rm LE}(I)=\mathbb{N}^n-\Delta$, 
    we say that $\Delta$ is the {\it standard set attached to $I$}.
    \item If $\Delta$ is a standard set, the set $\mathscr{C}(\Delta)$ of {\it corners} of $\Delta$ is the set of all 
    $\alpha\in\mathbb{N}^n-\Delta$ such that for all $i$, $\alpha-e_{i}\notin\mathbb{N}^n-\Delta$,
    where $e_{i}$ is the $i$-standard basis vector. 
    \item If $\Delta$ is a standard set, 
    the {\it border} of $\Delta$ is the set $\mathscr{B}(\Delta)=\cup_{i=1}^n(\Delta+e_{i})-\Delta$. 
    \item (Note that if $\Delta$ is a standard set, then $\Delta\cup\mathscr{B}(\Delta)$ is a standard set as well.)
    \item An {\it edge point} of a standard set $\Delta$ is an $\epsilon\in\Delta$ such that 
    there exist $\lambda$ and $\lambda^\prime$ in $\{e_{1},\ldots,e_{n}\}$ having the property that 
    $\epsilon+\lambda$ and $\epsilon+\lambda^\prime$ both lie in $\mathscr{B}(\Delta)$
    and $\epsilon+\lambda+\lambda^\prime$ lies in $\mathscr{C}(\Delta\cup\mathscr{B}(\Delta))$. 
    \item If $\epsilon\in\Delta$ is an edge point, the vectors $\lambda$ and $\lambda^\prime$ are, however, 
    not uniquely determined. Therefore, in the situation of the last bulleted item, 
    we call $(\epsilon,\lambda,\lambda^\prime)$ an {\it edge triple}.
  \end{itemize}
\end{dfn}

Figure \ref{fig} shows an example of a standard set in $\mathbb{N}^2$, along with its corners, its border and its edge points. 
The standard set is drawn in thick lines, the corners are marked by bullets, 
all other points in the border are marked by circles, 
and the edge points are marked by boxes.

\begin{center}
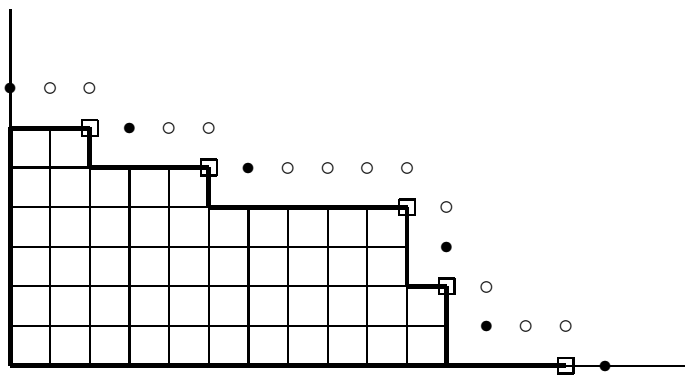
\begin{figure}[ht]
  \begin{picture}(270,155)
     % the rest of the coordinate axes
    \put(240,10){\line(1,0){45}}
    \put(30,100){\line(0,1){45}}
    % the grid inside the standard set
    \put(30,25){\line(1,0){165}}
    \put(30,40){\line(1,0){150}}
    \put(30,55){\line(1,0){150}}
    \put(30,70){\line(1,0){75}}
    \put(30,85){\line(1,0){30}}
    \put(45,10){\line(0,1){90}}
    \put(60,10){\line(0,1){75}}
    \put(75,10){\line(0,1){75}}
    \put(90,10){\line(0,1){75}}
    \put(105,10){\line(0,1){60}}
    \put(120,10){\line(0,1){60}}
    \put(135,10){\line(0,1){60}}
    \put(150,10){\line(0,1){60}}
    \put(165,10){\line(0,1){60}}
    \put(180,10){\line(0,1){30}}
    % border and corners
    \put(30,115){\circle*{4}}
    \put(45,115){\circle{4}}
    \put(60,115){\circle{4}}
    \put(75,100){\circle*{4}}
    \put(90,100){\circle{4}}
    \put(105,100){\circle{4}}
    \put(120,85){\circle*{4}}
    \put(135,85){\circle{4}}
    \put(150,85){\circle{4}}
    \put(165,85){\circle{4}}
    \put(180,85){\circle{4}}
    \put(195,70){\circle{4}}
    \put(195,55){\circle*{4}}
    \put(210,40){\circle{4}}
    \put(210,25){\circle*{4}}
    \put(225,25){\circle{4}}
    \put(240,25){\circle{4}}
    \put(255,10){\circle*{4}}
    % the edge points
    \put(57,103){\line(1,0){6}}
    \put(57,97){\line(1,0){6}}
    \put(57,97){\line(0,1){6}}
    \put(63,97){\line(0,1){6}}
    \put(102,88){\line(1,0){6}}
    \put(102,82){\line(1,0){6}}
    \put(102,82){\line(0,1){6}}
    \put(108,82){\line(0,1){6}}
    \put(177,73){\line(1,0){6}}
    \put(177,67){\line(1,0){6}}
    \put(177,67){\line(0,1){6}}
    \put(183,67){\line(0,1){6}}
    \put(192,43){\line(1,0){6}}
    \put(192,37){\line(1,0){6}}
    \put(192,37){\line(0,1){6}}
    \put(198,37){\line(0,1){6}}
    \put(237,13){\line(1,0){6}}
    \put(237,7){\line(1,0){6}}
    \put(237,7){\line(0,1){6}}
    \put(243,7){\line(0,1){6}}
    % the standard set
    \linethickness{.5mm}
    \put(30,10){\line(1,0){165}}
    \put(30,10){\line(0,1){90}}
    \put(180,40){\line(1,0){15}}
    \put(105,70){\line(1,0){75}}
    \put(60,85){\line(1,0){45}}
    \put(30,100){\line(1,0){30}}
    \put(60,85){\line(0,1){15}}
    \put(105,70){\line(0,1){15}}
    \put(180,40){\line(0,1){30}}
    \put(195,10){\line(0,1){30}}
    \put(195,10){\line(1,0){45}}
  \end{picture}
\caption{A standard set, $\circ$ its border, $\bullet$ its corners, and $\boxempty$ its edge points.}
\label{fig}
\end{figure}
\end{center}

The standard set of Figure \ref{fig} has the property that for each edge point $\epsilon$,
there exist unique $\lambda$ and $\lambda^\prime$ such that $(\epsilon,\lambda,\lambda^\prime)$ 
is an edge triple. 
Figure \ref{figedgetriples} shows a standard set $\Delta$, again drawn in thick lines, 
such that each edge point $\epsilon$ admits multiple $\lambda$ and $\lambda^\prime$
making an edge triple $(\epsilon,\lambda,\lambda^\prime)$. 
The border of $\Delta$, which is at the same time the set of corners of $\Delta$, is marked by bullets. 
The border of $\Delta\cup\mathscr{B}(\Delta)$, 
which is at the same time the set of corners of $\Delta\cup\mathscr{B}(\Delta)$, is marked by diamonds. 
The edge points are marked by boxes. 
The edge triples of $\Delta$ are presented in the table below. 

\begin{center}
  \begin{tabular}{|c|c|c|}
    \hline
    $\epsilon$ & $\lambda$ & $\lambda^\prime$ \\
    \hline
    (1,0,0) & (1,0,0) & (0,1,0) \\
    (1,0,0) & (1,0,0) & (0,0,1) \\
    (1,0,0) & (0,1,0) & (0,0,1) \\
    (0,1,0) & (0,1,0) & (1,0,0) \\
    (0,1,0) & (0,1,0) & (0,0,1) \\
    (0,1,0) & (1,0,0) & (0,0,1) \\
    (0,0,1) & (0,0,1) & (1,0,0) \\
    (0,0,1) & (0,0,1) & (0,1,0) \\
    (0,0,1) & (1,0,0) & (0,1,0) \\
    \hline
  \end{tabular}
\end{center}

\begin{center}
\begin{figure}
  \begin{picture}(120,135)
    % the rest of the coordinate system
    \put(75,50){\line(1,0){60}}
    \put(60,65){\line(0,1){60}}
    \put(49.8,43.2){\line(-3,-2){40}}
    % the border, at the same time the corners
    \put(90,50){\circle*{4}}
    \put(75,65){\circle*{4}}
    \put(60,80){\circle*{4}}
    \put(49.8,58.2){\circle*{4}}
    \put(39.8,36.6){\circle*{4}}
    \put(64.8,43.2){\circle*{4}}
    % the border of the border, at the same time the corners of the border
    \put(102.4,47.5){$\diamond$}
%    \put(105,50){\circle{4}}
    \put(87.4,62.5){$\diamond$}
%    \put(90,65){\circle{4}}
    \put(72.4,77.5){$\diamond$}
%    \put(75,80){\circle{4}}
    \put(57.4,92.5){$\diamond$}
%    \put(60,95){\circle{4}}
    \put(47.2,70.7){$\diamond$}
%    \put(49.8,73.2){\circle{4}}
    \put(37.2,49.1){$\diamond$}
%    \put(39.8,51.6){\circle{4}}
    \put(27.2,27.5){$\diamond$}
%    \put(29.8,30){\circle{4}}
    \put(52.2,34.1){$\diamond$}
%    \put(54.8,36.6){\circle{4}}
    \put(77.2,40.7){$\diamond$}
%    \put(79.8,43.2){\circle{4}}
    \put(62.2,55.7){$\diamond$}
%    \put(64.8,58.2){\circle{4}}
    % the edge points
    \put(72,53){\line(1,0){6}}
    \put(72,47){\line(1,0){6}}
    \put(72,47){\line(0,1){6}}
    \put(78,47){\line(0,1){6}}
    \put(57,68){\line(1,0){6}}
    \put(57,62){\line(1,0){6}}
    \put(57,62){\line(0,1){6}}
    \put(63,62){\line(0,1){6}}
    \put(46.9,46.3){\line(1,0){6}}
    \put(46.9,40.3){\line(1,0){6}}
    \put(46.9,40.3){\line(0,1){6}}
    \put(52.9,40.3){\line(0,1){6}}
    % the standard set
    \linethickness{.5mm}
    \put(60,50){\line(1,0){15}}
    \put(60,50){\line(0,1){15}}
    \put(60,50){\line(-3,-2){10}}
    \put(60,50.1){\line(-3,-2){10}}
    \put(60,50.2){\line(-3,-2){10}}
    \put(60,50.3){\line(-3,-2){10}}
    \put(60,50.4){\line(-3,-2){10}}
    \put(60,49.9){\line(-3,-2){10}}
    \put(60,49.8){\line(-3,-2){10}}
    \put(60,49.7){\line(-3,-2){10}}
    \put(60,49.6){\line(-3,-2){10}}
    \put(60,49.5){\line(-3,-2){10}}
    \put(60,49.4){\line(-3,-2){10}}
  \end{picture}
\caption{A standard set $\Delta$ together with $\bullet$ elements of $\mathscr{B}(\Delta)$ and 
$\diamond$ elements of $\mathscr{C}(\Delta\cup\mathscr{B}(\Delta))$.}
\label{figedgetriples}
\end{figure}
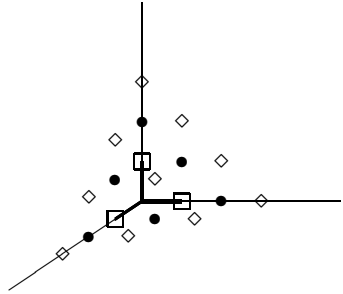
\end{center}

\begin{dfn}
  A {\it Gr\"obner basis} of an ideal $I$ is a finite subset $G$ of $I$ such that the ideals in $S$ 
  generated by ${\rm LT}(I)$ and ${\rm LT}(g)$, for $g\in G$, agree. 
  Note that not every ideal $I$ necessarily admits a Gr\"obner basis, 
  since $k$ was not assumed to be noetherian.
  A Gr\"obner basis $G$ is called {\it reduced} if 
  $\{{\rm LE}(g);\,g\in G\}=\mathscr{C}(\Delta)$, where ${\rm LE}(I)=\mathbb{N}^n-\Delta$; 
  each $g\in G$ is monic; 
  and all non-leading exponents of $g$ lie in the standard set attached to $I$. 
  An ideal $I$ admits a reduced Gr\"obner basis if, and only if, $I$ is monic. 
  (See \cite{aschenbrenner}, Theorem 2.11 and \cite{wibmer}, Theorem 4.)
\end{dfn}

%%%%%%%%%%%%%%%%%%%%%%%%%%%%%%%%%%%%%%

\section{The Hilbert functor of points and its standard subfunctors}\label{subfunctors}

Fix a positive integer $d$. We consider the {\it Hilbert functor of points} 
\begin{equation}\label{functor}
  \begin{split}
    \HHi^d_{S/k}:(k\text{-Alg})&\to(\text{Sets})\\
    B&\mapsto\left\lbrace
      \begin{array}{c}
        \phi:B[x]\to Q\text{ such that }\\
        \phi\text{ is surjective and }\\
        Q\text{ is a locally free $B$-module of rank $d$}
      \end{array}
    \right\rbrace/\sim\,,
  \end{split}
\end{equation}
where $\phi:B[x]\to Q$ and $\phi^\prime:B[x]\to Q^\prime$ are equivalent if there 
exists a $B$-algebra isomorphism $\psi:Q\to Q^\prime$ such that the following diagram commutes:
\begin{equation*}
  \begin{CD}
    B[x]@>\phi>>Q\\
    @V{\rm id}VV         @VV\psi V\\
    B[x]@>\phi^\prime>>Q^\prime\,.
  \end{CD}
\end{equation*}
Therefore $\phi$ and $\phi^\prime$ are equivalent if, and only if, their kernels agree. 
In this sense, the functor $\HHi^d_{S/k}$ parametrizes all ideals in the polynomial ring $S$ 
which are locally free of codimension $d$. 

At this point a remark on local freeness is in order. 
In the literature, one can find at least two definitions of when a $B$-module is locally free 
(see \cite{eisenbud}, p.137). 
The first is to demand that for each prime ideal $\mathfrak{p}\subset B$, 
the localized module $M_{\mathfrak{p}}$ is free over the localized ring $B_{\mathfrak{p}}$. 
The second is to demand that there exist $f_{1},\ldots,f_{t}\in B$ generating the unit ideal such that
each localization $M[f_{i}^{-1}]$ is a free $R[f_{i}^{-1}]$-module. 
The second definition (which is used for instance in \cite{haimanbernd}) is stronger.
However, if the module $M$ is locally free of a finite rank $d$, both definitions agree, 
and are equivalent to the following statement: For each prime ideal $\mathfrak{p}\subset B$, 
there exists an $f\in B-\mathfrak{p}$ such that the localized module $M_f$ is free over the localized ring $B_f$. 
We will use this definition in the proof of Proposition \ref{cover} below. 

Our first goal is to cover the functor $\HHi^d_{S/k}$ by a finite collection of open subfunctors,
indexed by all standard sets of size $d$. 
We shall now define these subfunctors.
Given a standard set $\Delta$, we use the shorthand notation 
$x^\Delta$ for the family $(x^\beta)_{\beta\in\Delta}$ and $kx^\Delta=\oplus_{\beta\in\Delta}kx^\beta$. 
We consider the canonical inclusion 
\begin{equation*}
    \iota_{\Delta}:kx^\Delta\to S\,.
\end{equation*}

\begin{dfn}\label{defdelta}
  Let $\Delta$ be a standard set of size $d$. 
  We define $\HHi^\Delta_{S/k}$ to be the subfunctor of $\HHi^d_{S/k}$ 
  which associates to each $k$-algebra $B$ the set of equivalence classes of all 
  $\phi:B[x]\to Q$ as in \eqref{functor} such that the composition
  \begin{equation*}
    \begin{CD}
      Bx^\Delta=B\otimes_{k}kx^\Delta
      @>{\rm id}\otimes\iota_{\Delta}>>B[x]@>\phi>>Q
    \end{CD}
  \end{equation*}
  is surjective, and therefore an isomorphism.
\end{dfn}

(The same functors are studied in \cite{huibregtse}, as was mentioned in the Introduction.)
In particular, all $Q$ appearing in $\HHi^\Delta_{S/k}(B)$ are free $B$-modules of rank $d$. 
Evidently there are the following alternative descriptions of the subfunctor,
\begin{equation}\label{delta1}
  \HHi^\Delta_{S/k}(B)=
  \left\lbrace
  \begin{array}{c}
      \phi:B[x]\to Q\text{ such that }\\
      \phi\text{ is surjective and }\\
      (x^\beta+\ker\phi)_{\beta\in\Delta}\text{ is a }B\text{-basis of }B[x]/\ker\phi
    \end{array}
  \right\rbrace/\sim\,,
\end{equation}
and also 
\begin{equation}\label{delta2}
  \HHi^\Delta_{S/k}(B)=
  \left\lbrace
  \begin{array}{c}
      \phi:B[x]\to Q\text{ such that }\\
      \phi\text{ is surjective and }\\
      (\phi(x^\beta))_{\beta\in\Delta}\text{ is a }B\text{-basis of }Q
    \end{array}
  \right\rbrace/\sim\,.
\end{equation}
Upon fixing an isomorphism $Q=Bx^\Delta$ and requiring that $\phi\circ({\rm id}\otimes\iota_{\Delta})={\rm id}$, 
we can rephrase the functor $\HHi^\Delta_{S/k}$ as follows:
\begin{equation}\label{delta3}
  \HHi^\Delta_{S/k}(B)=
  \left\lbrace
  \begin{array}{c}
    \phi:B[x]\to Bx^\Delta;\\
    \phi\text{ is a $k$-algebra homomorphism such that }\\
    \phi\circ({\rm id}\otimes\iota_{\Delta})={\rm id}\,.
  \end{array}
  \right\rbrace\,.
\end{equation}
The multiplicative structure on $Bx^\Delta$, making this module a $B$-algebra, 
is induced by that on $B[x]$ by the equation $B[x]/\ker\phi=Bx^\Delta$. 
Note that by fixing the isomorphism, we pick one representative of the equivalence class modulo $\sim$. 
In what follows, we will shift freely between the descriptions \eqref{delta1}, \eqref{delta2} and \eqref{delta3}. 

One can replace the homomorphism $\iota_{\Delta}$ by an arbitrary 
$k$-module homomorphism $\phi:k^d\to S$ and define a functor $\HHi^\phi_{S/k}$ analogous to the above. 
Such functors have been used in \cite{norge}, Section 5.1. 
The authors state that $\HHi^\phi_{S/k}$ is an open subfunctor of $\HHi^d_{S/k}$
and give a sketch of proof for this. 
For preparing the ground for the next sections, 
we carefully prove openness of the subfunctor $\HHi^\Delta_{S/k}$ of $\HHi^d_{S/k}$ here.
The proof for a functor $\HHi^\phi_{S/k}$ as in \cite{norge} is entirely analogous. 

\begin{lmm}\label{opensubfunctor}
  The canonical inclusion $i:\HHi^\Delta_{S/k}\to\HHi^d_{S/k}$ is an open embedding of functors. 
\end{lmm}

\begin{proof}
  Given a $k$-scheme $X$, we denote by $h_{X}$ the ${\rm Hom}$ functor which sends 
  a $k$-scheme $Y$ to the set ${\rm Mor}_{k}(Y,X)$. 
  By \cite{eisenbudharris}, Definition VI-5, we have to check that for each $k$-algebra $B$ 
  and each morphism of functors $\psi:h_{{\rm Spec}\,B}\to\HHi^d_{S/k}$, 
  the above horizontal arrow in the cartesian diagram 
  \begin{equation}\label{cartesian}
    \begin{CD}
      \mathscr{G}@>>>h_{{\rm Spec}\,B}\\
      @VVV @VV\psi V\\
      \HHi^\Delta_{S/k}@>i>>\HHi^d_{S/k}
    \end{CD}
  \end{equation}
  is isomorphic to the the inclusion of functors $h_{U}\to h_{{\rm Spec}\,B}$ 
  induced by the inclusion of schemes $U\to{\rm Spec}\,B$,
  where $U$ is an open subscheme of ${\rm Spec}\,B$. 
  So let an arrow $\psi:h_{{\rm Spec}\,B}\to\HHi^d_{S/k}$ be given. 
  By Yoneda's Lemma (see \cite{eisenbudharris}, Lemma VI-1), this is an element of $\HHi^d_{S/k}(B)$, 
  therefore the equivalence class of a surjective $\phi:B[x]\to Q$. 
  After localizing in $B$ at $f_{1},\ldots,f_{s}\in B$ which generate the unit ideal, 
  we may assume that $Q$ is a free $B$-module of rank $d$. 
  Further, let $\rho:k\to B$ be the structure morphism of the $k$-algebra $B$.
  The functor $\mathscr{G}$ in the cartesian diagram \eqref{cartesian}
  associates to each $k$-algebra $A$ the set of all pairs 
  $(g,h)$ in $h_{{\rm Spec}\,B}({\rm Spec}\,A)\times\HHi^\Delta_{S/k}(A)$ such that 
  $\psi(g)=i(h)$ in $\HHi^d_{S/k}$. 
  However, $g$ is nothing but a $k$-algebra homomorphism $\gamma:B\to A$, 
  and $h$ is nothing but (the equivalence class of) a $k$-algebra homomorphism $\eta:A[x]\to Q^\prime$. 
  Therefore, the condition $\psi(g)=i(h)$ says that the morphisms 
  \begin{equation*}
    \phi\otimes\gamma:A[x]\otimes_{B}A=B[x]\to Q\otimes_{B}A
  \end{equation*}
  and
  \begin{equation*}
    \eta:A[x]\to Q^\prime
  \end{equation*}
  are in the same equivalence class. 
  After localizing also at certain elements of $A$, we may assume that $Q^\prime$ is free of rank $d$. 
  We now fix isomorphisms $Q\otimes_{B}A=Ax^\Delta$ and $Q^\prime=Ax^\Delta$
  and accordingly demand that $\phi\otimes\gamma=\eta$. 
  Then the condition making the diagram cartesian is that 
  $\eta$ lies in $\HHi^\Delta_{S/k}(A)$. 
  In other words, we have reformulated the functor $\mathscr{G}$ as follows:
  $\mathscr{G}(A)$ is the set of all $\gamma:B\to A$ such that 
  $\phi\otimes\gamma:A[x]\to Ax^\Delta$ is an $A$-algebra homomorphism and 
  \begin{equation}\label{composition}
    (\phi\otimes\gamma)\circ(\iota_{\Delta}\otimes(\gamma\circ\rho)):Ax^\Delta\to A[x]\to Ax^\Delta
  \end{equation}
  is an isomorphism. Consider the special case $B=B$, $\gamma={\rm id}:B\to B$ and the composition 
  \begin{equation*}
    \phi\circ(\iota_{\Delta}\otimes\rho):Bx^\Delta\to B[x]\to Bx^\Delta\,.
  \end{equation*}
  Let $M$ be the matrix of this $B$-module homomorphism, and $J\subset B$ be the ideal 
  generated by $\det(M)$. Then clearly for any $\gamma:B\to A$, 
  the composition \eqref{composition} is an isomorphism if, and only if, $A=A\gamma(J)$. 
  By Exercise VI-6 of \cite{eisenbudharris}, we are done.
\end{proof}

As was mentioned in the Introduction, the statement of Lemma \ref{opensubfunctor} is implicit in \cite{huibregtse},
but not explicitly proved there. 

%%%%%%%%%%%%%%%%%%%%%%%%%%%%%%%%%%%%%%

\section{The standard cover}\label{standardcovering}

In Section 5.2 of \cite{norge}, the authors show with a very quick argument 
that their functors $\HHi^\phi_{S/k}$, 
where $\phi$ runs through all homomorphisms $B^d\to B[x]$, 
form an open cover of the functor $\HHi^d_{S/k}$, 
and also that there exists a finite set of such subfunctors which covers $\HHi^d_{S/k}$. 
We now show, in a constructive way, that our subfunctors $\HHi^\Delta_{S/k}$, 
which are also finite in number, suffice to cover $\HHi^d_{S/k}$.
Our covering family of subfunctors is a subfamily of the family of \cite{norge}, and is minimal. 

\begin{pro}\label{cover}
  The functors $\HHi^\Delta_{S/k}$, where $\Delta$ runs through all standard sets of size $d$, 
  form an open cover of the functor $\HHi^d_{S/k}$. 
  Moreover, this cover is minimal in the sense that when removing any member of it, 
  the result is no longer a cover. 
\end{pro}

\begin{proof}
  We will show that for all $B\in(k\text{-Alg})$, for all prime ideals $\mathfrak{p}\subset B$
  and for all $\phi\in\HHi^d_{S/k}(B)$, 
  there exist a $g\in B-\mathfrak{p}$ and a standard set $\Delta$ of size $d$ such that the localization 
  \begin{equation*}
    (\phi\otimes{\rm id}_{B_{g}})\circ(\iota_{\Delta}\otimes{\rm id}_{B_{g}}):B_{g}x^\Delta\to B_{g}[x]\to Q_{g}
  \end{equation*}
  is an isomorphism. 
  This will prove that the various $\HHi^\Delta_{S/k}$ cover $\HHi^d_{S/k}$. 

  Let $B$ be a $k$-algebra and $\phi:B[x]\to Q$ be a $B$-algebra homomorphism representing an element of 
  $\HHi^d_{S/k}(B)$, and let $\mathfrak{p}\subset B$ be a prime ideal. 
  We use the localization $B_{\mathfrak{p}}$ and its residue field 
  $\kappa=B_{\mathfrak{p}}/\mathfrak{p}B_{\mathfrak{p}}$. 
  Upon tensoring $\phi$ with $B_{\mathfrak{p}}$ and $\kappa$, respectively, we obtain the extensions
  \begin{equation}\label{extensions}
    \begin{split}
      \phi_{\mathfrak{p}}:B_{\mathfrak{p}}[x]&\to Q_{\mathfrak{p}}\,,\\
      \phi_{\kappa}:\kappa[x]&\to Q_{\kappa}\,.
    \end{split}
  \end{equation}
  By assumption, $Q$ is locally free of rank $d$, i.e., there exist an $f\in B-\mathfrak{p}$ such that
  $Q_{f}=\oplus_{i=1}^dB_{f}\epsilon_{i}$.
  Localizing further, we get $Q_{\mathfrak{p}}=\oplus_{j=1}^dB_{\mathfrak{p}}\epsilon_{j}$. 
  Taking residue classes, we get $Q_{\kappa}=\oplus_{j=1}^d\kappa\epsilon_{j}$.
  Local freeness of $Q$ and surjectivity of $\phi$ imply that both maps in \eqref{extensions} are surjective. 
  Since $\kappa$ is a field, the ideal $\ker\phi_{\kappa}$ has a Gr\"obner basis w.r.t. $\prec$, 
  with a standard set $\Delta$ attached to it. 
  As $Q_{\kappa}$ has dimension $d$, the standard set has size $d$.
  The family $x^\beta+\ker\phi_{\kappa}$, where $\beta$ runs through $\Delta$, 
  is a $\kappa$-basis of $\kappa[x]/\ker\phi_{\kappa}$. 
  Therefore the family $\phi_{\kappa}(x^\beta)$, where $\beta$ runs through $\Delta$, 
  is a $\kappa$-basis of $Q_{\kappa}$.
  From the commutative diagram
  \begin{equation*}
    \begin{CD}
      B_{\mathfrak{p}}[x]@>\phi_{\mathfrak{p}}>>Q_{\mathfrak{p}}\\
      @V{\rm can}VV @VV{\rm can}V\\
      \kappa[x]@>\phi_{\kappa}>>Q_{\kappa}\,,
    \end{CD}
  \end{equation*}
  where the vertical arrows are the canonical maps, we see that
  $\phi_{\mathfrak{p}}(x^\beta)$ is a lift of $\phi_{\kappa}(x^\beta)$
  w.r.t. the canonical map. Nakayama's Lemma (see \cite{eisenbud}, Corollary 4.8) implies that the family 
  $\phi_{\mathfrak{p}}(x^\beta)$, where $\beta$ runs through $\Delta$, 
  generates the $B_{\mathfrak{p}}$-module $Q_{\mathfrak{p}}$. 
  As the rank of $Q_{\mathfrak{p}}$ is $d=\#\Delta$, this family is even a $B_{\mathfrak{p}}$-basis.
  
  Therefore the composition 
  \begin{equation*}
    \phi_{\mathfrak{p}}\circ\iota_{\Delta}:B_{\mathfrak{p}}x^\Delta
    \to B_{\mathfrak{p}}[x]\to Q_{\mathfrak{p}}=\oplus_{i=1}^dB_{\mathfrak{p}}\epsilon_{i}
  \end{equation*}
  is an isomorphism. Going from left to right, we write the image of the basis element 
  $x^\gamma$ under the composition as
  \begin{equation}\label{there}
    (\phi_{\mathfrak{p}}\circ\iota_{\Delta})(x^\beta)=\sum_{i=1}^d\frac{c_{\beta,i}}{g_{\beta,i}}\epsilon_{i}\,.
  \end{equation}
  Going from right to left, we write the image of the basis element $\epsilon_{i}$ as
  \begin{equation}\label{back}
    (\phi_{\mathfrak{p}}\circ\iota_{\Delta})^{-1}(\epsilon_{i})
    =\sum_{\beta\in\Delta}\frac{d_{i,\beta}}{h_{i,\beta}}x^\beta\,.
  \end{equation}
  Here all $g_{\beta,i}$ and all $h_{i,\beta}$ lie in $B-\mathfrak{p}$. We set 
  \begin{equation*}
    h=(\prod_{\beta\in\Delta}\prod_{i=1}^dg_{\beta,i})\cdot
    (\prod_{i=1}^d\prod_{\beta\in\Delta}h_{i,\beta})
  \end{equation*}
  and $g=fh$. (Remember that $f$ is the element of $B-\mathfrak{p}$ with respect to which we localized earlier.)
  Then $B_{g}=(B_{f})_{h}$ and therefore $Q_{g}=\oplus_{i=1}^dB_{g}\epsilon_{i}$. 
  Formulas \eqref{there} and \eqref{back} define homomorphisms 
  \begin{equation*}
    B_{\mathfrak{p}}x^\Delta\to\oplus_{i=1}^dB_{\mathfrak{p}}\epsilon_{i}
  \end{equation*}
  and
  \begin{equation*}
    \oplus_{i=1}^dB_{\mathfrak{p}}\epsilon_{i}\to B_{\mathfrak{p}}x^\Delta\,,
  \end{equation*}
  resp., which are obviously inverses of each other. 
  The first assertion of the proposition is proved. 
    
  As for the second assertion, we fix a standard set $\Pi$ of size $d$ and 
  consider the ideal $I=(x^\alpha;\alpha\in\mathbb{N}^n-\Pi)\subset S$. 
  Then clearly $S/I$ is an element of $\HHi^\Pi_{S/k}$. 
  For all standard sets $\Delta\neq\Pi$ of the same size, 
  there exists an element $\beta\in\Pi-\Delta$. Therefore $x^\beta+I$ is zero in $S/I$, 
  and the family $(x^\beta+I)_{\beta\in\Delta}$ is not a $k$-basis of $S/I$. 
  It follows that the functors $\HHi^\Delta_{S/k}$, for all $\Delta\neq\Pi$, 
  do not suffice to cover all of $\HHi^d_{S/k}$.
\end{proof}

As was mentioned in the Introduction, the statement of Proposition \ref{cover} is implicit in \cite{huibregtse}, 
but not explicitly proved there. 

%%%%%%%%%%%%%%%%%%%%%%%%%%%%%%%%%%%%%%

\section{Gr\"obner bases in the standard subfunctors}\label{bases}

Let us further investigate the functor $\HHi^\Delta_{S/k}$. 
Let $B$ be a $k$-algebra, $\mathfrak{p}\subset B$ a prime ideal and $\phi\in\HHi^\Delta_{S/k}(B)$. 
In the course of the proof of Proposition \ref{cover}, 
we made use of polynomials lying in the ideal $\ker\phi_{\kappa}$. 
Since $\Delta$ is the standard set attached to the ideal $\ker\phi_{\kappa}$, 
each element of the reduced Gr\"obner basis of $\ker\phi_{\kappa}$ can be expressed as 
\begin{equation}\label{kernel}
  f_{\alpha}=x^\alpha+\sum_{\beta\in\Delta}c_{\alpha,\beta}x^\beta\,,
  \text{ where }c_{\alpha,\beta}=0\text{ if }\alpha\prec\beta\,.
\end{equation}
The latter condition guarantees that ${\rm LE}(f_{\alpha})=\alpha$. 
A priori a polynomial as in \eqref{kernel} exists only for all $\alpha\in\mathscr{C}(\Delta)$.
The collection $\{f_{\alpha};\alpha\in\mathscr{C}(\Delta)\}$ is the reduced Gr\"obner basis, which is unique. 
Therefore the polynomial of \eqref{kernel} is unique for all $\alpha\in\mathscr{C}(\Delta)$.
The following lemma (applied to $R=\kappa$, $I=\ker\phi_{\kappa}$) 
implies that a unique polynomial as in \eqref{kernel} exists for all $\alpha\in\mathbb{N}^n-\Delta$. 

\begin{lmm}\label{favouritelemma}
  Let $R$ be a ring and $\Delta$ a standard set. 
  Assume that for all $\xi\in\mathscr{C}(\Delta)$, there exists a monic $f_{\xi}\in R[x]$ such that
  ${\rm LE}(f_{\xi})=\xi$ and all non-leading exponents of $f_{\xi}$ lie in $\Delta$. 
  Define $I$ to be the ideal $(f_{\xi};\,\xi\in\mathscr{C}(\Delta))$ in $R[x]$.
  Then the following statements hold.
  \begin{enumerate}
    \item[(i)] For all $\alpha\in\mathbb{N}^n-\Delta$, there exists a unique $f_{\alpha}\in I$ 
      such that ${\rm LE}(f_{\alpha})=\alpha$ and all non-leading exponents of $f_{\alpha}$ lie in $\Delta$.
    \item[(ii)] All coefficients of all $f_{\alpha}$ are polynomial expressions with coefficients in $\mathbb{Z}$
      of ${\rm coef}(f_{\alpha},x^\beta)$, for $\alpha\in\mathscr{C}(\Delta)$, 
      $x^\beta\in{\rm supp}(f_{\alpha})$. 
    \item[(iii)] If ${\rm LE}(I)=\mathbb{N}^n-\Delta$, then $I$ is monic with reduced Gr\"obner basis
      $(f_{\xi})_{\xi\in\mathscr{C}(\Delta)}$. 
      Moreover, the family $(f_{\alpha})_{\alpha\in\mathbb{N}^n-\Delta}$ is an $R$-basis of the module $I$.
  \end{enumerate}
\end{lmm}

This lemma is apparently well-known, at least in the case where $R$ is field. 
However, it is hard to find a reference for it in the literature, 
as was mentioned in the discussion after Lemma 15 in \cite{dplanes}.
Its proof boils down to an inductive construction of the polynomials $f_{\alpha}$. 

We have seen that by Nakayama's Lemma the family of all $x^\beta$, where $\beta$ runs through $\Delta$, 
is a $B_{\mathfrak{p}}$-basis of $B_{\mathfrak{p}}[x]/\ker\phi_{\mathfrak{p}}$. 
Therefore each polynomial $f_{\alpha}\in\ker\phi_{\kappa}$ as in \eqref{kernel}, 
for $\alpha\in\mathbb{N}-\Delta$, has a unique lift to an element
\begin{equation*}
  \widehat{f}_{\alpha}=x^\alpha+\sum_{\beta\in\Delta}\widehat{c}_{\alpha,\beta}x^\beta
\end{equation*}
of $\ker\phi_{\mathfrak{p}}$. However, though $c_{\alpha,\beta}=0$ for $\alpha\prec\beta$, 
the coefficients $\widehat{c}_{\alpha,\beta}$ need not be zero for $\alpha\prec\beta$. 

\begin{pro}\label{grobner}
  The ideal $\ker\phi_{\mathfrak{p}}$ is monic with Gr\"obner basis $\widehat{f}_{\alpha}$, 
  for $\alpha\in\mathscr{C}(\Delta)$, if, and only if, 
  $\widehat{c}_{\alpha,\beta}=0$ for all $\alpha\in\mathscr{C}(\Delta)$ 
  and for all $\beta\in\Delta$ such that $\alpha\prec\beta$. 
\end{pro}

\begin{proof}
  This is a consequence of Lemma \ref{favouritelemma}. 
\end{proof}

In the complementary case, the set $\{f_{\alpha};\alpha\in\mathscr{B}(\Delta)\}$ is still the {\it border basis}
of $\ker\phi_{\mathfrak{p}}$, in the terminology of \cite{krbook}, Section 6.4.  
In our context, border bases are best described as follows. 
Take a $k$-algebra $B$ and a $\phi:B[x]\to Q$ in $\HHi^{\Delta}_{S/k}$. 
Let $\alpha\in\mathbb{N}^n-\Delta$, then by \eqref{delta1}, there exist unique $d_{\alpha,\beta}\in B$, 
for $\beta\in\Delta$, such that 
\begin{equation*}
  x^\alpha+\sum_{\beta\in\Delta}d_{\alpha,\beta}x^\beta=0\in B[x]/\ker\phi\,,
\end{equation*}
or equivalently, 
\begin{equation}\label{border}
  f_{\alpha}=x^\alpha+\sum_{\beta\in\Delta}d_{\alpha,\beta}x^\beta\in\ker\phi\,.
\end{equation}
The collection $\{f_{\alpha};\alpha\in\mathscr{B}(\Delta)\}$ is the border basis of $\ker\phi$. 
If in addition $\ker\phi$ is monic with standard set $\Delta$, 
then Lemma \ref{favouritelemma} implies that 
the collection $\{f_{\alpha};\alpha\in\mathscr{C}(\Delta)\}$ is the reduced Gr\"obner basis of $\ker\phi$.
In this sense the notion of border bases is a generalization of the notion of Gr\"obner bases. 
The goal of the next section is to exhibit that observation in the language of Hilbert functors.

%%%%%%%%%%%%%%%%%%%%%%%%%%%%%%%%%%%%%%

\section{The Gr\"obner subfunctors}\label{groebnersubfunctors}

In \cite{huibregtse}, Theorem 37, \cite{norge}, Theorem 5.4 and \cite{bertin}, Theorem 2.8, 
the authors show that the functor $\HHi^\Delta_{S/k}$ is representable by an affine scheme. 
We make use of this fact in this section, 
denoting by $\Hi^\Delta_{S/k}$ the representing scheme. 
(We will give explicit descriptions of the coordinate ring of this scheme in Sections \ref{affines} and \ref{examples}.)
Proposition \ref{grobner} suggests to consider the following elements of $\HHi^\Delta_{S/k}(B)$:

\begin{dfn}
  For each $k$-algebra $B$, let $\HHi^{\prec\Delta}_{S/k}(B)$ be the set of equivalence classes of 
  surjective $B$-algebra homomorphisms $\phi:B[x]\to Q$ such that $\ker\phi$ 
  has a reduced Gr\"obner basis of the form 
  \begin{equation}\label{eltsofkernel}
    f_{\alpha}=x^\alpha+\sum_{\beta\in\Delta,\,\beta\prec\alpha}d_{\alpha,\beta}x^\beta\,,
  \end{equation}
  where $\alpha$ runs through $\mathscr{C}(\Delta)$. 
\end{dfn}

As was mentioned in Section \ref{notation}, an ideal admits a reduced Gr\"obner basis if, and only if, it is monic. 
Moreover, the equivalence class of a surjective $B$-algebra homomorphism $\phi:B[x]\to Q$ is determined by its kernel, 
and as a monic ideal is determined by its reduced Gr\"obner basis. 
This gives us the following alternative characterizations of $\HHi^{\prec\Delta}_{S/k}(B)$ as: 
\begin{itemize}
  \item the set of equivalence classes of surjective $\phi:B[x]\to Q$ 
  such that $\ker\phi$ is a monic ideal with standard set $\Delta$. 
  \item the set of all monic ideals in $B[x]$ with standard set $\Delta$.
  \item the set of all reduced Gr\"obner bases in $B[x]$ with standard set $\Delta$.
\end{itemize}

\begin{lmm}\label{subfunctor}
  $\HHi^{\prec\Delta}_{S/k}$ is a subfunctor of $\HHi^{\Delta}_{S/k}$. 
\end{lmm}

\begin{proof}
  Let $\phi:B[x]\to Q$ be an element of $\HHi^{\prec\Delta}_{S/k}(B)$. 
  The division algorithm (see \cite{cox}, Section 2, \S3) shows that the family $(x^\beta+\ker\phi)$, 
  where $\beta$ runs through $\Delta$, is a $B$-basis of $B[x]/\ker\phi$. 
  Therefore the family $\phi(x^\beta)$, where $\beta$ runs through $\Delta$, is a $B$-basis of $Q$. 
  Hence $\phi:B[x]\to Q$ is also an element of $\HHi^{\Delta}_{S/k}(B)$. 
  In particular, we may assume that $Q=Bx^\Delta$. 
  
  We show that $\HHi^{\prec\Delta}_{S/k}$ is a functor. Let
  \begin{equation*}
    \phi:B[x]\to Bx^\Delta
  \end{equation*}
  be an element of $\HHi^{\prec\Delta}_{S/k}(B)$ and
  $\psi:B\to A$ be a $k$-algebra homomorphism. 
  Tensoring is right exact, hence a surjective homomorphism
  \begin{equation*}
    \phi\otimes{\rm id}:A[x]\to Ax^\Delta\,.
  \end{equation*}
  We have to show that $\ker\phi\otimes{\rm id}$ is monic with standard set $\Delta$. For this, 
  we write the elements of the reduced Gr\"obner basis of $\ker\phi$ as in formula \eqref{eltsofkernel}.
  We define
  \begin{equation*}
    g_{\alpha}=x^\alpha+\sum_{\beta\in\Delta,\,\beta\prec\alpha}\psi(d_{\alpha,\beta})x^\beta\,,
  \end{equation*}
  for all $\alpha\in\mathscr{C}(\Delta)$. 
  Then clearly all $g_{\alpha}$ lie in $\ker(\phi\otimes{\rm id})$. By Lemma \ref{favouritelemma} (i), 
  we get a unique polynomial of the form 
  \begin{equation*}
    g_{\alpha}=x^\alpha+\sum_{\beta\in\Delta,\,\beta\prec\alpha}e_{\alpha,\beta}x^\beta
  \end{equation*}
  even for all $\alpha\in\mathbb{N}^n-\Delta$, 
  and in particular, all these $g_{\alpha}$ lie in $\ker(\phi\otimes{\rm id})$. 
  Now let $g$ be an arbitrary element of $\ker(\phi\otimes{\rm id})$. 
  Denote the leading term of $g$ by $cx^\mu$. 
  We have to show that $\mu$ lies in $\mathbb{N}^n-\Delta$, 
  as in this case, Lemma \ref{favouritelemma} (iii) guarantees that 
  $\ker(\phi\otimes{\rm id})$ is monic with standard set $\Delta$.
  Consider the polynomial 
  \begin{equation*}
    g^\prime=g-\sum_{\beta\in\mathbb{N}^n-\Delta,\,\beta\prec\mu}{\rm coef}(g,x^\beta)g_{\beta}\,.
  \end{equation*}
  Then $g^\prime$ lies in $\ker(\phi\otimes{\rm id})$;
  its support is contained in $\Delta\cup\{\mu\}$; and its leading term is $cx^\mu$. 
  However, as $\phi\otimes{\rm id}$ lies in $\HHi^{\prec\Delta}_{S/k}(A)$, 
  we know that the family $x^\beta+\ker(\phi\otimes{\rm id})$, 
  where $\beta$ runs through $\Delta$, is a basis of $A[x]/\ker(\phi\otimes{\rm id})$. 
  This shows that if $\mu\in\Delta$, then $c=0$, a contradiction. 
  Hence $\mu\in\mathbb{N}^n-\Delta$. 
\end{proof}

\begin{lmm}\label{zariskisheaf}
  $\HHi^{\prec\Delta}_{S/k}$ is a Zariski sheaf. 
\end{lmm}

\begin{proof}
  Let $B$ be a $k$-Algebra, $(U_{i}={\rm Spec}\,B_{g_{i}})_{i\in I}$ an open cover of ${\rm Spec}\,B$
  by distinguished open sets and $\phi_{i}\in\HHi^{\prec\Delta}_{S/k}(B_{g_{i}})$ such that for all $i,j$, 
  \begin{equation*}
    \phi_{i}\otimes{\rm id}:B_{g_{i}}\otimes_{B_{g_{i}}}B_{g_{j}}\to Q_{i}\otimes_{B_{g_{i}}}B_{g_{j}}
  \end{equation*}
  and 
  \begin{equation*}
    \phi_{j}\otimes{\rm id}:B_{g_{j}}\otimes_{B_{g_{j}}}B_{g_{i}}\to Q_{j}\otimes_{B_{g_{j}}}B_{g_{i}}
  \end{equation*}
  agree, i.e., define the same map
  \begin{equation*}
    \phi_{ij}:B_{g_{i}g_{j}}\to Q_{ij}=B_{g_{i}g_{j}}x^\Delta\,.
  \end{equation*}
  We write the elements of the reduced Gr\"obner basis of $\ker\phi_{i}$ and $\ker\phi_{j}$, resp., as
  \begin{equation*}
    \begin{split}
      f^{(i)}_{\alpha}&=x^\alpha+\sum_{\beta\in\Delta,\,\beta\prec\alpha}d^{(i)}_{\alpha,\beta}x^\beta\,,\\
      f^{(j)}_{\alpha}&=x^\alpha+\sum_{\beta\in\Delta,\,\beta\prec\alpha}d^{(j)}_{\alpha,\beta}x^\beta\,,
    \end{split}
  \end{equation*}
  resp., where $\alpha$ runs through $\mathscr{C}(\Delta)$. 
  From Lemma \ref{subfunctor} we know that 
  $\ker\phi_{ij}=\ker\phi_{i}\otimes{\rm id}=\ker\phi_{j}\otimes{\rm id}$ 
  is monic with standard set $\Delta$. 
  The images of $f^{(i)}_{\alpha}$ and $f^{(j)}_{\alpha}$, resp., in $B_{g_{i}g_{j}}[x]$ have the following properties:
  \begin{itemize}
    \item They lie in $\ker\phi_{ij}$. 
    \item Their leading exponent is $\alpha$.
    \item Their non-leading exponents lie in $\Delta$. 
  \end{itemize}
  Therefore they are the reduced Gr\"obner basis of $\ker\phi_{ij}$. 
  In particular, $f^{(i)}_{\alpha}$ and $f^{(j)}_{\alpha}$ agree on ${\rm Spec}\,B_{g_{i}g_{j}}$. 
  The sheaf axiom for the quasi-coherent sheaf $B[x]\,\,\widetilde{}\,$ on ${\rm Spec}\,B$ 
  provides a polynomial $f_{\alpha}\in B[x]$ whose image in $B_{g_{i}}[x]$ is $f^{(i)}_{\alpha}$ for all $i$. 
  It is clear that this polynomial takes the shape \eqref{eltsofkernel}. 
  Upon defining $I=(f_{\alpha};\alpha\in\mathscr{C}(\Delta))$ and $\phi:S\to Q=S/I$ to be the canonical map, 
  we have lifted the various homomorphisms $\phi_{i}$ to a homomorphism $\phi$. 
  The same line of arguments as at the end of the proof of Lemma \ref{subfunctor} shows that
  $I$ is monic with Gr\"obner basis $f_{\alpha}$, where $\alpha$ runs through $\mathscr{C}(\Delta)$. 
  Therefore $\phi$ lies in $\HHi^{\prec\Delta}_{S/k}(B)$.  
\end{proof}

\begin{thm}\label{locallyclosed1}
  $\HHi^{\prec\Delta}_{S/k}$ is represented by a closed subscheme 
  $\Hi^{\prec\Delta}_{S/k}$of $\Hi^{\Delta}_{S/k}$.
\end{thm}

\begin{proof}
  We prove this by applying Proposition 2.9 of \cite{haimanbernd}. 
  For this we adopt two items of the terminology of the cited paper. 
  \begin{itemize}
    \item Let $B$ be an object of $(k\text{-Alg})$, and let a condition on morphisms 
      $\psi:B\to A$ in $(k\text{-Alg})$ be given. 
      We say that the condition is {\it closed} if there exists an ideal $J\subset B$ such that $\psi:B\to A$
      satisfies the condition if, and only if, $\psi$ factors through the canonical map $B\to B/J$. 
    \item Let $B$ be an object of $(k\text{-Alg})$ and the $B$-algebra homomorphism $\phi:B[x]\to Q$ 
      be an object of $\HHi^{\Delta}_{S/k}(B)$. 
      We say that a morphism $\psi:B\to A$ in $(k\text{-Alg})$ satisfies $V_{B,\phi}$
      if the $A$-algebra homomorphism $\HHi^{\Delta}_{S/k}(\psi)(\phi)$, 
      which is an element of $\HHi^{\Delta}_{S/k}(A)$, lies in $\HHi^{\prec\Delta}_{S/k}(A)$.
  \end{itemize}
  
  By Proposition 2.9 of \cite{haimanbernd}, the functor $\HHi^{\prec\Delta}_{S/k}$ 
  (which is a Zariski sheaf by Lemma \ref{zariskisheaf})
  is represented by a closed subscheme of $\HHi^{\Delta}_{S/k}$ if, and only if, 
  for all $B$ in $(k\text{-Alg})$ and all $\phi:B[x]\to Q$ in $\HHi^{\Delta}_{S/k}(B)$,
  the condition $V_{B,\phi}$ is closed. 
  
  Let $B$ and $\phi$ as above be given. 
  Then the family $(x^\beta+\ker\phi)_{\beta\in\mathscr{C}(\Delta)}$ 
  is a $B$-basis of $B[x]/\ker\phi$. By Lemma \ref{favouritelemma}, 
  there is a unique polynomial of the form 
  \begin{equation*}
    f_{\alpha}=x^\alpha+\sum_{\beta\in\Delta}d_{\alpha,\beta}x^\beta\in\ker\phi
  \end{equation*}
  for all $\alpha\in\mathbb{N}^n-\Delta$. Define $J\subset B$ to be the ideal generated by all $d_{\alpha,\beta}$, 
  where $\alpha$ runs through $\mathbb{N}^n-\Delta$ and $\beta$ runs through all elements of $\Delta$ 
  such that $\alpha\prec\beta$. 
  
  Given a morphism $\psi:B\to A$ in $(k\text{-Alg})$, 
  the homomorphism $\HHi^{\prec\Delta}_{S/k}(\psi)(\phi)$ is nothing but the tensor product
  $\phi\otimes{\rm id}:A[x]\to Q\otimes_{B}A$. 
  The polynomial
  \begin{equation*}
    \psi(f_{\alpha})=x^\alpha+\sum_{\beta\in\Delta}\psi(d_{\alpha,\beta})x^\beta
  \end{equation*}
  is the unique element of $\ker(\phi\otimes{\rm id})$ such that 
  its leading exponent is $\alpha$ and all non-leading exponents lie in $\Delta$. 
  Now $\psi:B\to A$ factors through $B\to B/J$ if, and only if, 
  for all $\alpha\in\mathscr{C}(\Delta)$ and all $\beta\in\Delta$ such that $\alpha\prec\beta$, 
  we have $\psi(d_{\alpha,\beta})=0$. 
  This is equivalent to the ideal $(\psi(f_{\alpha});\,\mathscr{C}(\Delta))\subset A[x]$ 
  being monic with reduced Gr\"obner basis $\{\psi(f_{\alpha});\alpha\in\mathscr{C}(\Delta)\}$. 
  Therefore $\psi:B\to A$ factors through $B\to B/J$ if, and only if, 
  $\ker(\phi\otimes{\rm id})$ is monic with Gr\"obner basis $\psi(f_{\alpha})$, 
  where $\alpha$ runs through $\mathscr{C}(\Delta)$. 
  We have proved that $V_{B,\phi}$ is a closed condition. 
\end{proof}

%%%%%%%%%%%%%%%%%%%%%%%%%%%%%%%%%%%%%%

\section{The Gr\"obner strata}\label{decomposition}

We have proved that in the chain of inclusion \eqref{chain}
the first inclusion is a closed immersion and the second inclusion is an open immersion. 

\begin{dfn} 
  We call the locally closed subscheme $\Hi^{\prec\Delta}_{S/k}$ of $\Hi^{d}_{S/k}$ the 
  {\rm Gr\"obner stratum attached to the standard set $\Delta$}. 
\end{dfn}

Here is an example illustrating the difference between $\Hi^{\Delta}_{S/k}$ and $\Hi^{\prec\Delta}_{S/k}$. 

\begin{ex}\label{triangle}
  Let $k=\mathbb{Z}$, $\Delta=\{0,e_1,e_2\}\subset\mathbb{N}^2$ 
  and $\prec$ the lexicographic order on $S=k[x_1,x_2]$ such that $x_1\succ x_2$. The ideal 
  \begin{equation*}
    \begin{split}
      I_a=(&x_1^2+2x_1+2x_2+3\,,\\
      &x_1x_2+2x_1+2x_2+3\,,\\
      &{\bf x_2^2+2x_1+2x_2+3})\subset S
    \end{split}
  \end{equation*}
  lies in $\Hi^{\Delta}_{S/k}$ but not in $\Hi^{\prec\Delta}_{S/k}$. The ideal 
  \begin{equation*}
    \begin{split}
      I_b=(&x_1^2+2x_1+2x_2+3\,,\\
      &x_1x_2+2x_1+2x_2-4\,,\\
      &{\bf x_2^2+2x_2+2})\subset S
    \end{split}
  \end{equation*}
  lies in $\Hi^{\prec\Delta}_{S/k}$. 
  The difference stems from the coefficient of the term $x_1$ in the generators printed in boldface type. 
  That coefficient vanishes only for $I_b$ and not for $I_a$. 
  The given generators of $I_b$ are the reduced Gr\"obner basis of $I_b$. 
  The ideal $I_a$ does not have a reduced Gr\"obner basis if $\prec$ is the term order we chose. 
  However, if we replace that order by the graded lexicographic order, 
  then the given generators are the reduced Gr\"obner basis of $I_a$. 
  We will justify this in Example \ref{triangleagain}, see Section \ref{examples} below. 
\end{ex}

Gr\"obner strata and related objects have been studied by many authors, 
see \cite{evain1}, \cite{evain2} or \cite{altmannbernd}. 
The cited authors refer to these schemes as {\it Schubert schemes}, or {\it Schubert cells}. 
Their terminology is motivated by the analogy of the inclusion 
$\Hi^{\prec\Delta}_{S/k}\subset\Hi^{d}_{S/k}$ to the inclusion of a Schubert cell in the Grassmannian 
in the case where $\Delta$ is a subset of the standard basis $\{e_{1},\ldots,e_{n}\}\subset\mathbb{N}^n$, 
augmented by $0\in\mathbb{N}^n$. 
One interesting thing about Gr\"obner strata is the following statement.

\begin{thm}\label{coprod}
  As a topological space, the scheme $\Hi^{d}_{S/k}$ decomposes into locally closed strata as follows,
  \begin{equation}\label{stratahd}
    \Hi^{d}_{S/k}=\coprod_{\Delta}\Hi^{\prec\Delta}_{S/k}\,,
  \end{equation}
  where the disjoint union goes over all standard sets $\Delta\subset\mathbb{N}^n$ of size $d$. 
\end{thm}

\begin{proof}
  We have to show that each closed point of $\Hi^{d}_{S/k}$ lies in precisely one stratum $\Hi^{\prec\Delta}_{S/k}$. 
  Let $x:{\rm Spec}\,F\to\Hi^{d}_{S/k}$ be a closed point, $F$ a field. 
  We interpret this as an element of of $\Hi_{S/k}(F)$, i.e., as a surjective $F$-algebra homomorphism $\phi:F[x]\to Q$.
  The kernel of this homomorphism has a well-defined reduced Gr\"obner basis, 
  and a well-defined standard set $\Delta$. Therefore $x$ lies in $\Hi^{\prec\Delta}_{S/k}$, 
  and not in any $\Hi^{\prec\Pi}_{S/k}$, for $\Pi\neq\Delta$.
\end{proof}

Note that in general a non-closed point of $\Hi^{d}_{S/k}$ 
does not lie in any stratum $\Hi^{\prec\Delta}_{S/k}$ of \eqref{stratahd}. 
Indeed, a non-closed point of $\Hi^{d}_{S/k}$ lies in some $\Hi^{\Delta}_{S/k}$, 
and thus corresponds to a homomorphism $\phi:B[x]\to Q$ lying in $\HHi^{\Delta}_{S/k}(B)$, 
where $B$ is a ring rather than a field. 
In particular, for all $\alpha\in\mathscr{B}(\alpha)$, there exist $f_\alpha$ as in \eqref{border}. 
However, it may happen that $d_{\alpha,\beta}\neq0$ for some pair $\alpha\prec\beta$. 
If so, then the point does not lie in $\Hi^{\prec\Delta}_{S/k}$. 

\begin{ex}
  Consider the case $n=2$, $d=3$. There are three standard sets
  \begin{equation*}
    \Delta_1=\{0,e_1,2e_1\}\,,\,\Delta_2=\{0,e_1,e_2\}\,,\,\Delta_3=\{0,e_2,2e_2\}
  \end{equation*}
  of size $3$ in $\mathbb{N}^2$. In Examples \ref{triangleagain} and \ref{axis} below we will show that 
  the three corresponding open patches $\Hi^{\Delta_i}_{S/k}$ are all isomorphic to $\mathbb{A}^6$. 
  More precisely, $ \Hi^{\Delta_i}_{S/k}={\rm Spec}\,R^{\Delta_i}$, where
  \begin{equation*}
    \begin{split}
      R^{\Delta_1}&=k[T_{(3,0),\beta},T_{(0,1),\beta};\,\beta\in\Delta_1]\,,\\
      R^{\Delta_2}&=k[T_{(2,0),\beta},T_{(1,1),\beta},T_{(0,2),\beta};\,\beta\in\Delta_2-\{0\}]\,,\\
      R^{\Delta_3}&=k[T_{(1,0),\beta},T_{(3,0),\beta};\,\beta\in\Delta_1]\,.
    \end{split}
  \end{equation*}
  Let $K_i={\rm Frac}\,R^{\Delta_i}$ and $\eta_i={\rm Spec}\,K_i$ be the generic point of $\Hi^{\Delta_i}_{S/k}$. 
  The point $\eta_i$ of $\Hi^{\Delta_i}_{S/k}$ corresponds to the $K_i$-algebra $Q_i=K_i[x_1,x_2]/J_i$, where 
  \begin{equation*}
    \begin{split}
      J_1=(&x_1^3-T_{(3,0),(2,0)}x_1^2-T_{(3,0),(1,0)}x_1-T_{(3,0),(0,0)},\\
      &x_2-T_{(0,1),(2,0)}x_1^2-T_{(0,1),(1,0)}x_1-T_{(0,1),(0,0)})\,,\\
      J_2=(&x_1^2-T_{(2,0),(1,0)}x_1-T_{(2,0),(0,1)}x_2-T_{(2,0),(0,0)}\,\\
      &x_1x_2-T_{(1,1),(1,0)}x_1-T_{(1,1),(0,1)}x_2-T_{(1,1),(0,0)}\,\\
      &x_2^2-T_{(0,2),(1,0)}x_1-T_{(0,2),(0,1)}x_2-T_{(0,2),(0,0)})\,,\\
      J_3=(&x_1-T_{(1,0),(0,2)}x_2^2-T_{(1,0),(0,1)}x_2-T_{(1,0),(0,0)}\,,\\
      &x_2^3-T_{(0,3),(0,2)}x_2^2-T_{(0,3),(0,1)}x_2-T_{(0,3),(0,0)})\,.
    \end{split}
  \end{equation*}
  (The meaning of those $T_{\alpha,\beta}$ appearing in $J_i$ which are not generators of $R^{\Delta_i}$ 
  will be explained in Examples \ref{triangleagain} and \ref{axis} below.)
  We may assume that the term order satisfies $x_1\succ x_2$, thus in particular $x_1^3\succ x_2$. 
  Since $T_{(3,0),(0,1)}\neq0$ in $K_1$, we see that the given generators of $J_1$ are not a reduced Gr\"obner basis. 
  In fact, $J_1$ does not admit a reduced Gr\"obner basis. 
  Consequently $\eta_1$ does not lie in any stratum of \eqref{stratahd}. 
  Similarly, if $\prec$ is the lexicographic order, then $\eta_2$ does not lie in any stratum of \eqref{stratahd} either, 
  and if $\prec$ is the graded lexicographic order, then $\eta_3$ does lie in any stratum of \eqref{stratahd} either. 
\end{ex}

%%%%%%%%%%%%%%%%%%%%%%%%%%%%%%%%%%%%%%

\section{Representing the functors}\label{affines}

We start this section by briefly reviewing the construction of the affine scheme $\Hi^{\Delta}_{S/k}$ 
given in \cite{norge} and \cite{bertin}, Theorem 2.8. 
$\HHi^{\Delta}_{S/k}(B)$ is the set of equivalence classes of 
$B$-algebra homomorphisms $\phi:B[x]\to Q$
such that the composition $Bx^\Delta\to B[x]\to Q$ is an isomorphism. 
Each equivalence class of $\phi:B[x]\to Q$ corresponds to precisely one $B$-algebra structure on the $B$-module $Bx^\Delta$. 
Therefore $\HHi^{\Delta}_{S/k}(B)$ is reinterpreted as the set of all 
$B$-algebra homomorphisms $\phi:B[x]\to Bx^\Delta$
such that $\phi\circ(\iota_{\Delta}\otimes{\rm id}):Bx^\Delta\to Bx^\Delta$ is the identity map. 
Now that we have free modules with bases, 
we identify $\phi$ with its matrix $(a_{\alpha\beta})_{\alpha\in\mathbb{N}^n\,,\,\beta\in\Delta}$,
which is given by 
\begin{equation}\label{amatrix}
  \phi(x^\alpha)=\sum_{\beta\in\Delta}a_{\alpha\beta}x^\beta\,,\text{ for all }\alpha\in\mathbb{N}^n\,.
\end{equation}
The condition $\phi\circ(\iota_{\Delta}\otimes{\rm id})={\rm id}$ says that
\begin{equation*}
  a_{\alpha\beta}=\delta_{\alpha\beta}\,,\text{ for all }\alpha\in\mathbb{N}^n\text{ and for all }\beta\in\Delta\,.
\end{equation*}
The $B$-module homomorphism $\phi$ is a $B$-algebra homomorphism if, and only if, it is multiplicative. 
This characterization will be used in the proof of Proposition \ref{representing} below. 
For the time being, we use another characterization:
The $B$-module homomorphism $\phi$ is a $B$-algebra homomorphism if, and only if, its kernel is an ideal in $B[x]$. 
It is easy to check that the family $x^\alpha-(\iota_{\Delta}\otimes{\rm id})\circ\phi(x^\alpha)$, 
where $x^\alpha$ runs through all monomials in $B[x]$, generates the $B$-module $\ker\phi$. 
Therefore $\ker\phi$ is an ideal in $B[x]$ if, and only if, 
\begin{equation*}
  \phi(x^\lambda(x^\alpha-(\iota_{\Delta}\otimes{\rm id})\circ\phi(x^\alpha)))=0\,,
  \text{ for all }\lambda,\alpha\in\mathbb{N}^n\,.
\end{equation*}
Upon expressing $\phi$ by its matrix and using the fact that 
$\iota_{\Delta}$ is the canonical inclusion, this condition reads as follows:
\begin{equation*}
  \sum_{\beta\in\Delta}(a_{\lambda+\alpha,\beta}
  -\sum_{\gamma\in\Delta}a_{\alpha,\gamma}a_{\lambda+\gamma,\beta})x^\beta=0\,,
  \text{ for all }\lambda,\alpha\in\mathbb{N}^n\,.
\end{equation*}
Since $Bx^\Delta$ is free with basis $x^\Delta$, this means that 
\begin{equation}\label{structural}
  a_{\lambda+\alpha,\beta}-\sum_{\gamma\in\Delta}a_{\alpha,\gamma}a_{\lambda+\gamma,\beta}=0\,,
  \text{ for all }\lambda,\alpha\in\mathbb{N}^n\text{ and for all }\beta\in\Delta\,.
\end{equation}
Clearly it suffices to let $x^\lambda$ run only through $x_{1},\ldots,x_{n}$. 
Therefore the functor $\HHi^{\Delta}_{S/k}$ is represented by the affine scheme
\begin{equation}\label{infinitepresentation}
  \Hi^{\Delta}_{S/k}={\rm Spec}\,R/I^\Delta\,,
\end{equation}
where $I^\Delta$ is the ideal 
\begin{equation*}
  \begin{split}
    I^\Delta&=(T_{\alpha,\beta}-\delta_{\alpha,\beta};\alpha\in\mathbb{N}^n,\beta\in\Delta)\\
    &+(T_{\lambda+\alpha,\beta}-\sum_{\gamma\in\Delta}T_{\alpha,\gamma}T_{\lambda+\gamma,\beta};
    \alpha\in\mathbb{N}^n,\lambda\in\{e_{1},\ldots,e_{n}\},\beta\in\Delta)
  \end{split}
\end{equation*}
in the polynomial ring $R=k[T_{\alpha,\beta};\alpha\in\mathbb{N}^n,\beta\in\Delta]$. 

The heart of the above described method for obtaining the coordinate ring of the scheme 
$\Hi^{\Delta}_{S/k}$ is the system of equations \eqref{structural}. 
These are the {\it structural equations} defining the multiplicative structure on the $B$-algebra $Bx^\Delta$. 
In contrast to this approach to the coordinate ring of $\Hi^{\Delta}_{S/k}$, 
the same ring is obtained by using a border basis variant of Buchberger's $S$-pair criterion in the articles 
\cite{huibregtse}, \cite{kk1}, \cite{kk2}, \cite{kkr}, \cite{krarticle} and \cite{robbiano}.
In those articles, finite presentations of the coordinate ring of $\Hi^{\Delta}_{S/k}$ are given. 
At the moment our approach seems weaker, as the presentation of \eqref{infinitepresentation} 
uses infinitely many generators and relations. 
In the next section we will see that our approach is in fact stronger. 
In Proposition \ref{representing} we derive a finite presentation of the coordinate ring $R/I$. 
In Theorem \ref{fewer} below we will derive an improvement on this proposition. 

We introduce the following notation. 
If $N\subset\mathbb{N}^n$ is a standard set, we write $N^{(1)}=\mathscr{B}(N)$ and, for all $i\geq1$, 
$N^{(i+1)}=\mathscr{B}(N\cup N^{(1)}\cup\ldots\cup N^{(i)})$. 

\begin{pro}\label{representing}
  Let $N$ be a standard set in $\mathbb{N}^n$ containing $\Delta$. 
  Then the functor $\HHi^{\Delta}_{S/k}$ is represented by the affine scheme 
  $\Hi^{\Delta}_{S/k}={\rm Spec}\,R^\Delta$, where
  \begin{equation*}
    R^\Delta=R/I^\Delta
  \end{equation*}
  and $I^\Delta=I^\Delta_{1}+I^\Delta_{2}+I^\Delta_{3}$ is the sum of the ideals
  \begin{equation*}
    \begin{split}
      I^\Delta_{1}&=(T_{\alpha,\beta}-\delta_{\alpha,\beta};\alpha\in N,\beta\in\Delta)\,,\\
      I^\Delta_{2}&=(T_{\alpha+\lambda,\beta}-\sum_{\gamma\in\Delta}T_{\alpha,\gamma}T_{\gamma+\lambda,\beta};\\
      &\alpha\in N\cup N^{(1)},\lambda\in\{e_{1},\ldots,e_{n}\}
      \text{ s.t. }\alpha+\lambda\in N\cup N^{(1)},\beta\in\Delta)\text{ and }\\
      I^\Delta_{3}&=(\sum_{\gamma\in\Delta}T_{\alpha,\gamma}T_{\gamma+\lambda,\beta}
      -\sum_{\gamma\in\Delta}T_{\alpha^\prime,\gamma}T_{\gamma+\lambda^\prime,\beta};\\
      &\alpha,\alpha^\prime\in N^{(1)},\lambda,\lambda^\prime\in\{e_{1},\ldots,e_{n}\}
      \text{ s.t. }\alpha+\lambda=\alpha^\prime+\lambda^\prime\in N^{(2)},\beta\in\Delta)
    \end{split}
  \end{equation*}
  in the polynomial ring $R=k[T_{\alpha,\beta};\alpha\in N\cup N^{(1)},\beta\in\Delta]$. 
\end{pro}

\begin{proof}
  Again we start with a $B$-module homomorphism $\phi:B[x]\to Bx^\Delta$, 
  represented by its matrix $(a_{\alpha,\beta})$ as in \eqref{amatrix}. 
  Our goal is to find constraints on the coefficients $a_{\alpha,\beta}$ 
  which guarantee that $Bx^\Delta$ has a multiplicative structure
  such that $\phi$ is a $B$-algebra homomorphism. 
  As was mentioned above, $\phi$ is a $B$-algebra homomorphism if, and only if, $\phi$ is multiplicative. 
  By linearity of $\phi$, this is equivalent to $\phi(x^{\alpha+\beta})=\phi(x^\alpha)\phi(x^\beta)$ 
  for all $\alpha,\beta\in\mathbb{N}^n$, 
  and by an easy induction argument, the latter condition is equivalent to 
  \begin{equation}\label{mult}
    \phi(x^{\alpha+\lambda})=\phi(x^\alpha)\phi(x^\lambda)
  \end{equation}
  for all $\alpha\in\mathbb{N}^n$ and all $\lambda\in\{e_{1},\ldots,e_{n}\}$. 
  Therefore, our goal is to find constraints on the coefficients $a_{\alpha,\beta}$
  which guarantee that $Bx^\Delta$ has a multiplicative structure and the multiplicativity condition \eqref{mult} 
  holds for all $\alpha\in\mathbb{N}^n$ and all $\lambda\in\{e_{1},\ldots,e_{n}\}$.
  We will see in the course of the proof that we do not need the full matrix 
  $(a_{\alpha,\beta})_{\alpha\in\mathbb{N}^n,\beta\in\Delta}$, but rather only the rows indexed by $\alpha\in N\cup N^{(1)}$. 
  
  {\it Step 1.} As above, the condition $\phi\circ(\iota_{\Delta}\otimes{\rm id})={\rm id}$ translates into the
  following constraints on the coefficients $a_{\alpha,\beta}$:
  \begin{equation}\label{cond1}
    \forall\alpha\in N\,,\forall\beta\in\Delta\,:
    a_{\alpha,\beta}=\delta_{\alpha,\beta}\,.
  \end{equation}
  
  {\it Step 2.} We impose the multiplicativity condition \eqref{mult} 
  on all $\alpha$ and $\alpha+\lambda$ which lie in $N\cup N^{(1)}$. 
  Let us translate this into equations for the coefficients $a_{\alpha,\beta}$. 
  The left hand side of \eqref{mult} is
  \begin{equation*}
    \phi(x^{\alpha+\lambda})=\sum_{\beta\in\Delta}a_{\alpha+\lambda,\beta}x^\beta\,.
  \end{equation*}
  The right hand side of \eqref{mult} is a priori not defined before we have the multiplicative structure of $Bx^\Delta$ at hand. 
  However, upon assuming that \eqref{mult} holds for elements of $N\cup N^{(1)}$, 
  we can surmount that obstacle by a trick:
  \begin{equation*}
    \begin{split}
      \phi(x^\alpha)\phi(x^\lambda)&=\sum_{\gamma\in\Delta}a_{\alpha,\gamma}x^\gamma\phi(x^\lambda)
      =\sum_{\gamma\in\Delta}a_{\alpha,\gamma}\phi(x^\gamma)\phi(x^\lambda)\\
      &=\sum_{\gamma\in\Delta}a_{\alpha,\gamma}\phi(x^{\gamma+\lambda})
      =\sum_{\beta,\gamma\in\Delta}a_{\alpha,\gamma}a_{\gamma+\lambda,\beta}x^\beta\,.
    \end{split}
  \end{equation*}
  Here we used the fact that $x^\gamma=\phi(x^\gamma)$ if $\gamma\in\Delta$,
  and the multiplicativity condition \eqref{mult} for $\gamma$ and $\gamma+\lambda$ lying in $N\cup N^{(1)}$. 
  The two expressions have to coincide, hence the following constraints on the coefficients $a_{\alpha,\beta}$:
  \begin{equation}\label{cond2}
    \begin{split}
      &\forall\alpha\in N\cup N^{(1)}\,,\forall\lambda\in\{e_{1},\ldots,e_{n}\}
      \text{ s.t. }\alpha+\lambda\in N\cup N^{(1)}\,,\forall\beta\in\Delta:\\
      &a_{\alpha+\lambda,\beta}=\sum_{\gamma\in\Delta}a_{\alpha,\gamma}a_{\gamma+\lambda,\beta}
    \end{split}
  \end{equation}
  Note that these are just (some of) the structural equations \eqref{structural}. 

  At this point {\it multiplicativity holds within $N\cup N^{(1)}$}
  in the sense that \eqref{mult} holds if $\alpha,\alpha+\lambda\in N\cup N^{(1)}$. 

  {\it Step 3.} We define more values of $\phi$ by means of the equation \eqref{mult}.
  More precisely, we take $\alpha\in N^{(1)}$ and $\lambda\in\{e_{1},\ldots,e_{n}\}$ 
  such that $\alpha+\lambda\in N^{(2)}$ and define
  \begin{equation*}
    \phi(x^{\alpha+\lambda})=\sum_{\beta,\gamma\in\Delta}a_{\alpha,\gamma}a_{\gamma+\lambda,\beta}x^\beta\,.
  \end{equation*}
  Then the multiplicativity condition \eqref{mult} holds for these values of $\alpha$ and $\lambda$, 
  as the right hand side of the last equation is
  \begin{equation*}
    \begin{split}
      &\sum_{\beta,\gamma\in\Delta}a_{\alpha,\gamma}\phi(x^{\gamma+\lambda})
      =\sum_{\beta,\gamma\in\Delta}a_{\alpha,\gamma}\phi(x^\gamma)\phi(x^\lambda)\\
      =& \sum_{\beta,\gamma\in\Delta}a_{\alpha,\gamma}x^\gamma\phi(x^\lambda)
      =\phi(x^\alpha)\phi(x^\lambda)\,.
    \end{split}
  \end{equation*}
  At this point we have to make sure that the definition just given is unambiguous. 
  This means that if $\alpha^\prime\in N^{(1)}$ and $\lambda^\prime\in\{e_{1},\ldots,e_{n}\}$ are such that
  $\alpha+\lambda=\alpha^\prime+\lambda^\prime$, 
  the definitions of $\phi(x^{\alpha+\lambda})$ and of $\phi(x^{\alpha^\prime+\lambda^\prime})$ coincide. 
  This translates into the following constraints on the coefficients $a_{\alpha,\beta}$:
  \begin{equation}\label{cond3}
    \begin{split}
      &\forall\alpha,\alpha^\prime\in N^{(1)}\,,\forall\lambda,\lambda^\prime\in\{e_{1},\ldots,e_{n}\}
      \text{ s.t. }\alpha+\lambda=\alpha^\prime+\lambda^\prime\in N^{(2)}\,,\\
      &\forall\beta\in\Delta:
      \sum_{\gamma\in\Delta}a_{\alpha,\gamma}a_{\gamma+\lambda,\beta}
      =\sum_{\gamma\in\Delta}a_{\alpha^\prime,\gamma}a_{\gamma+\lambda^\prime,\beta}\,.
    \end{split}
  \end{equation}

  At this point {\it multiplicativity also holds when passing from $N^{(1)}$ to $N^{(2)}$}
  in the sense that \eqref{mult} holds if $\alpha\in N^{(1)}$ and $\alpha+\lambda\in N^{(2)}$. 
  
  {\it Step 4.} We claim that {\it multiplicativity holds within $N\cup N^{(1)}\cup N^{(2)}$}
  in the sense that \eqref{mult} holds if $\alpha,\alpha+\lambda\in N\cup N^{(1)}\cup N^{(2)}$. 

  We only have to check that if both $\alpha$ and $\alpha+\lambda$ lie in $N^{(2)}$. 
  In this case there exists some $\nu\in\{e_{1},\ldots,e_{n}\}$ such that $\alpha+\lambda-\nu\in N^{(1)}$. 
  In particular, $\nu\neq\lambda$. This implies that also $\alpha-\nu$ lies in $\mathbb{N}^n$, 
  as that element arises from $\alpha+\lambda$ 
  by subtracting two different standard basis elements, $\lambda$ and $\nu$. 
  Therefore $\alpha-\nu$ in fact lies in $N^{(1)}$. We obtain
  \begin{equation*}
    \phi(x^{\alpha+\lambda})=\phi(x^{\alpha+\lambda-\nu})\phi(x^\nu)
    = \phi(x^{\alpha-\nu})\phi(x^\lambda)\phi(x^\nu)=\phi(x^\alpha)\phi(x^\lambda)\,,
  \end{equation*}
  as desired. Here we used the fact that multiplicativity holds when passing from $N^{(1)}$ to $N^{(2)}$
  for the outer two equalities and the fact that multiplicativity holds within $N\cup N^{(1)}$ 
  for the inner equality. 
  
  {\it Step 5.} We define more values of $\phi$ in analogy to Step 3:
  We take $\alpha\in N^{(1)}$ and $\lambda,\mu\in\{e_{1},\ldots,e_{n}\}$ such that
  $\alpha+\lambda\in N^{(2)}$ and $\alpha+\lambda+\mu\in N^{(3)}$ and define
  \begin{equation*}
    \phi(x^{\alpha+\lambda+\mu})=\sum_{\beta,\gamma,\gamma^\prime\in\Delta}
    a_{\alpha,\gamma^\prime}a_{\gamma^\prime+\lambda,\gamma}a_{\gamma+\mu,\beta}x^\beta\,.
  \end{equation*}
  This definition makes the identity
  \begin{equation}\label{specialmult}
    \phi(x^{\alpha+\lambda+\mu})=\phi(x^{\alpha+\lambda})\phi(x^\mu)
  \end{equation}
  hold. We claim that the definition just given is unambiguous, 
  i.e. that if $\alpha^\prime\in N^{(1)}$ and $\lambda^\prime,\mu^\prime\in\{e_{1},\ldots,e_{n}\}$ are such that
  $\alpha+\lambda+\mu=\alpha^\prime+\lambda^\prime+\mu^\prime$, 
  the corresponding definitions of $\phi(x^{\alpha+\lambda})$ coincide. 
  For verifying this, we first note that we may assume that $\mu\neq\mu^\prime$, 
  since otherwise unambiguity is trivial. 
  By the same argument as in Step 4, 
  we see that $\alpha+\lambda-\mu^\prime=\alpha^\prime+\lambda^\prime-\mu$ lies in $\mathbb{N}^n$. 
  Therefore $\alpha+\lambda-\mu^\prime$ in fact lies in $N\cup N^{(1)}\cup N^{(2)}$. 
  Now we distinguish three cases. 
  
  {\it Case a.} $\lambda\neq\mu^\prime$. 
  Then by the same argument once more, $\alpha-\mu^\prime$ lies in $\mathbb{N}^n$. 
  It follows that $\alpha-\mu^\prime$ lies in $N\cup N^{(1)}\cup N^{(2)}$. 
  Therefore
  \begin{equation*}
    \begin{split}
    &\phi(x^{\alpha+\lambda+\mu})
    =\phi(x^{\alpha+\lambda})\phi(x^\mu)
    =\phi(x^\alpha)\phi(x^\lambda)\phi(x^\mu)\\
    &=\phi(x^{\alpha-\mu^\prime})\phi(x^{\mu^\prime})\phi(x^\lambda)\phi(x^\mu)
    =\phi(x^{\alpha-\mu^\prime+\lambda})\phi(x^{\mu^\prime})\phi(x^\mu)\\
    &=\phi(x^{\alpha^\prime-\mu+\lambda^\prime})\phi(x^{\mu^\prime})\phi(x^\mu)
    =\phi(x^{\alpha^\prime+\lambda^\prime})\phi(x^{\mu^\prime})
    =\phi(x^{\alpha^\prime+\lambda^\prime+\mu^\prime})\,.
    \end{split}
  \end{equation*}
  Here we used \eqref{specialmult} for the outer equalities and 
  the fact that multiplicativity holds within $N\cup N^{(1)}\cup N^{(2)}$ for all other equalities. 
  
  {\it Case b.} $\lambda^\prime\neq\mu$. 
  This is the same as the previous with the roles of primed and non-primed elements interchanged.
  
  {\it Case c.} $\lambda=\mu^\prime$ and $\lambda^\prime=\mu$. Then $\alpha=\alpha^\prime$. 
  As $\alpha\in N^{(1)}$, there exists a $\nu\in\{e_{1},\ldots,e_{n}\}$ such that $\alpha-\nu\in N$. 
  Again we get a chain of equalities: 
  \begin{equation*}
    \begin{split}
    &\phi(x^{\alpha+\lambda+\mu})
    =\phi(x^{\alpha+\lambda})\phi(x^{\lambda^\prime})
    =\phi(x^\alpha)\phi(x^\lambda)\phi(x^{\lambda^\prime})\\
    &=\phi(x^{\alpha-\nu})\phi(x^\nu)\phi(x^\lambda)\phi(x^{\lambda^\prime})
    =\phi(x^{\alpha+\lambda^\prime-\nu})\phi(x^\nu)\phi(x^\lambda)\\
    &=\phi(x^{\alpha+\lambda^\prime})\phi(x^\lambda)
    =\phi(x^{\alpha+\lambda^\prime+\mu^\prime})\,.
    \end{split}
  \end{equation*}
  Again we used \eqref{specialmult} for the outer equalities and 
  the fact that multiplicativity holds within $N\cup N^{(1)}\cup N^{(2)}$ for all other equalities.

  At this point {\it multiplicativity also holds when passing from $N^{(2)}$ to $N^{(3)}$} in the obvious sense. 

  {\it Induction Step.} Note that the proof of multiplicativity given in Step 4 was completely formal,
  only using the fact that multiplicativity holds within $N\cup N^{(1)}$ and when passing from $N^{(1)}$ to $N^{(2)}$. 
  Therefore, now that we know that multiplicativity holds within $N\cup N^{(1)}\cup N^{(2)}$
  and when passing from $N^{(2)}$ to $N^{(3)}$, 
  we can imitate Step 4 and prove that multiplicativity holds within $N\cup N^{(1)}\cup N^{(2)}\cup N^{(3)}$. 
  Analogously, Step 5 was completely formal and can be imitated for proving that 
  multiplicativity holds when passing from $N^{(3)}$ to $N^{(4)}$. 
  Then we imitate Step 4 again and prove that multiplicativity holds within 
  $N\cup N^{(1)}\cup N^{(2)}\cup N^{(3)}\cup N^{(4)}$, 
  imitate Step 5 again for proving that multiplicativity holds when passing from $N^{(4)}$ to $N^{(5)}$, and so on.
  This proves that the multiplicativity condition \eqref{mult} holds for all 
  $\alpha\in\mathbb{N}^n$ and all $\lambda\in\{e_{1},\ldots,e_{n}\}$. 
  
  {\it End of proof.} We see that the $B$-module homomorphism $\phi:B[x]\to Bx^\Delta$ 
  defines a $B$-algebra homomorphism if, and only if, the coefficients $a_{\alpha,\beta}$, 
  where $\alpha\in N\cup N^{(1)}$ and $\beta\in\Delta$, 
  satisfy the three conditions \eqref{cond1}, \eqref{cond2} and \eqref{cond3} of Steps 1, 2 and 3, resp.
  Therefore, an element $\phi$ of $\HHi^\Delta_{S/k}(B)$ is uniquely determined by the choice of elements $a_{\alpha,\beta}\in B$,
  for all $\alpha\in N\cup N^{(1)}$ and all $\beta\in\Delta$, 
  such that \eqref{cond1}, \eqref{cond2} and \eqref{cond3} hold. 
  That choice corresponds to the choice of a $k$-algebra homomorphism $R^\Delta\to B$. 
\end{proof}

The set $N$ of Proposition \ref{representing} can be chosen finite. 
Therefore the scheme $\Hi^{\Delta}_{S/k}$ is of finite type over $k$, 
and embedded as the closed subscheme corresponding to $I^\Delta$ into affine space with coordinates $T_{\alpha,\beta}$, 
for $\alpha\in N$, $\beta\in\Delta$. 
In view of the summand $I^\Delta_{1}$ of $I^\Delta$, we see that we only need the coordinates $T_{\alpha,\beta}$, 
for $\alpha\in N-\Delta$, $\beta\in\Delta$, for the ambient space. 
The smallest possible $N$ is $\Delta$, 
hence a closed immersion of $\Hi^{\Delta}_{S/k}$ into affine space of dimension $\#\mathscr{B}(\Delta)\#\Delta$. 
The same immersion is studied in
\cite{huibregtse}, \cite{kk1}, \cite{kk2}, \cite{kkr}, \cite{krarticle} and \cite{robbiano}. 
However, these articles do not use the matrices of $\phi(x^\alpha)$ but rather the polynomials 
\begin{equation*}
  f_{\alpha}=x^\alpha+\sum_{\beta\in\Delta}d_{\alpha,\beta}x^\beta\in\ker\phi
\end{equation*}
(cf. \eqref{eltsofkernel}). 
These polynomials carry the same information as our matrix, as
\begin{equation}\label{ftomatrix}
  a_{\alpha,\beta}=
    \begin{cases} 
      \,\,\,\delta_{\alpha,\beta}&\text{ if }\alpha\in\Delta\,,\\
      -d_{\alpha,\beta}&\text{ if }\alpha\in\mathbb{N}^n-\Delta\,,
    \end{cases}
\end{equation}
The work with the polynomials $f_{\alpha}$ makes syzygy criteria necessary in the cited articles.
The two summands $I^\Delta_{2}$ and $I^\Delta_{3}$ in our ideal $I^\Delta$ correspond to the concepts of 
{\it next-door-neighbors} and {\it across-the-street neighbors}, resp., in \cite{kk1}, Definition 17 and \cite{krarticle}, Section 4. 
We will say more about that in the next section, in which we dispose of many of the generators of $I^\Delta_{3}$.
Now we focus on the Gr\"obner functors.

\begin{cor}\label{prime}
  Let $N$, $R$ and $I^\Delta$ be as in Proposition \ref{representing}. 
  Then the functor $\HHi^{\prec\Delta}_{S/k}$ is represented by the affine scheme
  $\Hi^{\prec\Delta}_{S/k}={\rm Spec}\,R^{\prec\Delta}$, where 
  \begin{equation*}
    R^{\prec\Delta}=R/I^{\prec\Delta}
  \end{equation*}
  and 
  \begin{equation*}
    I^{\prec\Delta}=I^\Delta+(T_{\alpha,\beta};\alpha\in N\cup N^{(1)},\beta\in\Delta,\alpha\prec\beta)\,.
  \end{equation*}
\end{cor}

\begin{proof}
  The additional conditions defining $I^{\prec\Delta}$ express 
  the constraints on the subfunctor $\HHi^{\prec\Delta}_{S/k}$ of $\HHi^{\Delta}_{S/k}$
  which we discussed in the proof of Theorem \ref{locallyclosed1}, 
  in terms of the variables $T_{\alpha,\beta}$. 
\end{proof}

Equivalently, the scheme $\Hi^{\prec\Delta}_{S/k}$ is the closed subscheme of $\Hi^{\Delta}_{S/k}$
defined by the ideal in $R^\Delta$ generated by the images of all $T_{\alpha,\beta}$, 
for $\alpha\in N\cup N^{(1)}$ and $\beta\in\Delta$ such that $\alpha\prec\beta$.
Note that for the ideal defining $\Hi^{\prec\Delta}_{S/k}$, the identity 
\begin{equation*}
  I^{\prec\Delta}=I^\Delta+(T_{\alpha,\beta};\alpha\in\mathscr{C}(\Delta),\beta\in\Delta,\alpha\prec\beta)
\end{equation*}
holds. In other words, many of the additional conditions defining $I^{\prec\Delta}$ follow from a few basic ones. 
(This is a consequence of Lemma \ref{favouritelemma}, applied to the ring $R^\Delta$ of Proposition \ref{representing}.)

By Corollary \ref{prime}, $\Hi^{\prec\Delta}_{S/k}$ is a closed subscheme 
of an affine space of dimension $\#\mathscr{B}(\Delta)\#\Delta$. 
We now cut down further the dimension of the ambient space. 
For this we consider the polynomial ring 
$R_{m}=k[T_{\alpha,\beta};\alpha\in\mathscr{C}(\Delta),\beta\in\Delta,\alpha\succ\beta]$. 
For all $\alpha\in\mathscr{B}(\Delta)-\mathscr{C}(\Delta)$ and all $\beta\in\Delta$ such that $\alpha\succ\beta$, 
we define elements $T_{\alpha,\beta}$ of $R_{m}$ by recursion over $\alpha$ as follows: 
\begin{itemize}
  \item We start with $\Gamma=\mathscr{B}(\Delta)-\mathscr{C}(\Delta)$. 
  \item While $\Gamma\neq\emptyset$, we take the minimal element $\alpha$ of $\Gamma$, 
    we find some $\nu\in\{e_{1},\ldots,e_{n}\}$ such that $\alpha-\nu\in\mathscr{B}(\Delta)$, we define 
    \begin{equation}\label{restrictedsum}
      T_{\alpha,\beta}=\sum_{\gamma\in\Delta,\alpha\succ\gamma+\nu\succ\beta}
      T_{\alpha-\nu,\gamma}T_{\gamma+\nu,\beta}\,,
    \end{equation}
    where $T_{\gamma+\nu,\beta}=\delta_{\gamma+\nu,\beta}$ if $\gamma+\nu\in\Delta$, 
    and we replace $\Gamma$ by $\Gamma-\{\alpha\}$. 
\end{itemize}
Moreover, we define an ideal $I^{\prec\Delta}_{m}\subset R_{m}$ 
by the same formulas as the ideal $I^\Delta\subset R$ in Proposition \ref{representing}, 
applied in the case where $N=\Delta$, with the following modification:
In all summands of all generators of $I^\Delta$,
we replace all $T_{\alpha,\beta}$ such that $\alpha\in\Delta$ by $\delta_{\alpha,\beta}$, 
and delete all $T_{\alpha,\beta}$ such that $\alpha\succ\beta$. 

\begin{cor}\label{min}
  With the above notation, 
  we have $\Hi^{\prec\Delta}_{S/k}={\rm Spec}\,R^{\prec\Delta}_{m}$, where 
  \begin{equation*}
    R^{\prec\Delta}_{m}=R_{m}/I^{\prec\Delta}_{m}\,.
  \end{equation*}
\end{cor}

\begin{proof}
  We first check that our recursion is well-defined. 
  Indeed, for all $\alpha\in \mathscr{B}(\Delta)-\mathscr{C}(\Delta) $, 
  the existence of a standard basis element $\nu$ such that $\alpha-\nu\in\mathscr{B}(\Delta)$ 
  follows directly from the definitions. 
  If $\alpha$ is the object of consideration in one particular step of the recursion, 
  the element $\alpha-\nu$ either lies in $\mathscr{C}(\Delta)$ 
  or has been the object of consideration in an earlier stage of the recursion, as $\alpha\succ\alpha-\nu$. 
  In both cases $T_{\alpha-\nu,\gamma}$ is a well-defined element of $R_m$. Moreover, 
  in the sum \eqref{restrictedsum} we only consider those $\gamma\in\Delta$ for which $\alpha-\nu\succ\gamma$, 
  or equivalently, $\alpha\succ\gamma+\nu$. 
  Therefore the element $\gamma+\nu$ either lies in $\Delta$ 
  or has been the object of consideration in an earlier stage of the recursion. 
  In both cases $T_{\gamma+\nu,\beta}$ is a well-defined element of $R_m$. 
  
  Now we return to the notation of Corollary \ref{prime} in the case where $N=\Delta$ 
  and consider the rings $R$ and $R^{\prec\Delta}=R/I^{\prec\Delta}$ defined there. 
  For simplicity we write $T_{\alpha,\beta}$ for the images of the variables $T_{\alpha,\beta}\in R$ 
  in the quotient $R^{\prec\Delta}$. 
  In particular, $T_{\alpha,\beta}=0$ for all $\alpha\in\mathscr{B}(\Delta)$ and all $\beta\in\Delta$ 
  such that $\alpha\prec\beta$. Furthermore, 
  the presence of the summand $I^\Delta_{2}$ in the ideal $I^\Delta$ implies that whenever 
  $\alpha-\nu$ and $\alpha$ lie in $\mathscr{B}(\Delta)$, the identity
  \begin{equation*}
    T_{\alpha,\beta}=\sum_{\gamma\in\Delta}T_{\alpha-\nu,\gamma}T_{\gamma+\nu,\beta}
  \end{equation*}
  holds in $R^{\prec\Delta}$. 
  However, only those $\gamma\in\Delta$ for which $\alpha-\nu\succ\gamma$ and $\gamma+\nu\succ\beta$ 
  make a contribution to that sum. 
  This explains the definition of $T_{\alpha,\beta}$ given in \eqref{restrictedsum}. 
  As for the definition of $I^{\prec\Delta}_{m}$, 
  the replacement $T_{\alpha,\beta}=\delta_{\alpha,\beta}$ for all $\alpha\in\Delta$ 
  is clear from the presence of the summand $I^\Delta_{1}$ in the ideal $I^\Delta$; 
  and deleting all $T_{\alpha,\beta}$ such that $\alpha\succ\beta$ 
  stems from the equality $T_{\alpha,\beta}=0$ in $R^{\prec\Delta}$. 
  The assertion follows from Corollary \ref{prime}. 
\end{proof}

In the corollary we embedded $\Hi^{\prec\Delta}_{S/k}$ into an affine space of dimension 
\begin{equation*}
  p=\#\{(\alpha,\beta)\in\mathscr{C}(\Delta)\times\Delta;\alpha\succ\beta\}\,.
\end{equation*} 
Note that $p$ depends on both the shape of $\Delta$ and the term order $\prec$. 
Example \ref{axis} below shows that for special shapes of $\Delta$, 
there exists a term order $\prec$ such that this upper bound is sharp. 
(For the standard set of Example \ref{axis}, 
take $\prec$ to be the lexicographic order such that $x_1\prec x_i$, for all $i>1$.)
This observation motivates the subscript in the ideal $I^{\prec\Delta}_{m}$, which stands for {\it minimal}. 
However, it is not clear if minimality holds in a strict sense: 

\begin{q}
  Given a standard set $\Delta$, does there exist a term order $\prec$ such that 
  $p$ is the minimal dimension of an affine space into which $\Hi^{\prec\Delta}_{S/k}$ can be embedded? 
\end{q}

\begin{q}
  Given a term order $\prec$, does there exist a standard set $\Delta$ such that 
  $p$ is the minimal dimension of an affine space into which $\Hi^{\prec\Delta}_{S/k}$ can be embedded? 
\end{q}

In Proposition \ref{representing} we embedded $\Hi^{\Delta}_{S/k}$ into an affine space of dimension 
\begin{equation*}
  q=d\#\mathscr{B}(\Delta)\,.
\end{equation*} 
(This dimension is obtained when letting $N=\Delta$.)
For standard sets of size $d=1$, we trivially have $\Hi^{\Delta}_{S/k}=\mathbb{A}^n$, 
thus the dimension of $\Hi^{\Delta}_{S/k}$ equals $q$. 
It is not clear what happens for larger $d$:

\begin{q}
  For which $d\in\mathbb{N}$ does there exist a standard set $\Delta$ such that 
  $q$ is the minimal dimension of an affine space into which all $\Hi^{\Delta}_{S/k}$ can be embedded?
\end{q}

For special shapes of $\Delta$, the number $q$ is certainly not the minimal dimension which one can reach. 
A class of counterexamples is given by Example \ref{axis} again
(for the term order $\prec$ we considered above, we have $\Hi^{\Delta}_{S/k}=\Hi^{\prec\Delta}_{S/k}$).
Another class of counterexamples is given by Corollary 7.3.2 of \cite{huibregtsea2}, 
which states that if $n=2$ and $\Delta\subset\mathbb{N}^2$ has a ``sawtooth'' form depicted in Figure \ref{sawtooth}
(for any parameters $a$, $b$ and $c$), then $\Hi^{\Delta}_{S/k}$ is an affine space of dimension $2d$. 

\begin{center}
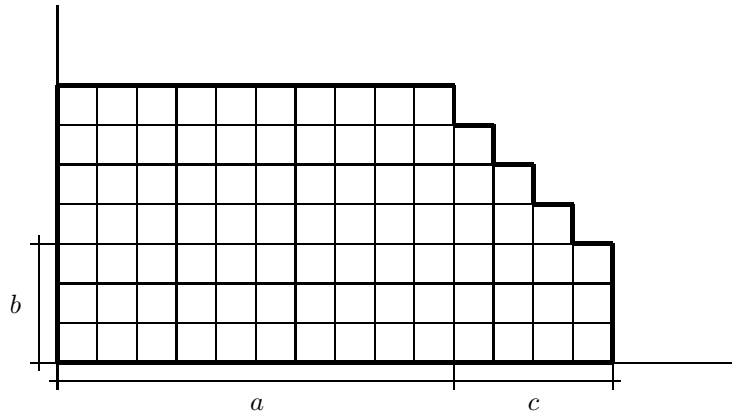
\begin{figure}[ht]
  \begin{picture}(270,175)
     % the rest of the coordinate axes
    \put(240,30){\line(1,0){45}}
    \put(30,135){\line(0,1){30}}
    % the grid inside the standard set
    \multiput(30,45)(0,15){2}{\line(1,0){210}}
    \put(30,75){\line(1,0){195}}
    \put(30,90){\line(1,0){180}}
    \put(30,105){\line(1,0){165}}
    \put(30,120){\line(1,0){150}}
    \multiput(45,30)(15,0){9}{\line(0,1){105}}
    \put(180,30){\line(0,1){90}}
    \put(195,30){\line(0,1){75}}
    \put(210,30){\line(0,1){60}}
    \put(225,30){\line(0,1){45}}
    % the delimiters
    \put(30,30){\line(-1,0){10}}
    \put(30,75){\line(-1,0){10}}
    \put(23,27){\line(0,1){51}}
    \put(30,30){\line(0,-1){10}}    
    \put(180,30){\line(0,-1){10}}    
    \put(240,30){\line(0,-1){10}}    
    \put(27,23){\line(1,0){216}}
    \put(103,12){\small $a$}
    \put(12,49){\small $b$}
    \put(208,12){\small $c$}
    % the standard set
    \linethickness{.5mm}
    \put(30,30){\line(1,0){210}}
    \put(30,30){\line(0,1){105}}
    \put(30,135){\line(1,0){150}}
    \put(240,30){\line(0,1){45}}
    \multiput(180,120)(15,-15){4}{\line(1,0){15}}
    \multiput(180,120)(15,-15){4}{\line(0,1){15}}
  \end{picture}
\caption{A standard set of sawtooth form}
\label{sawtooth}
\end{figure}
\end{center}

If $N$ is strictly larger than $\Delta$, 
the dimension of the ambient space of $\Hi^{\Delta}_{S/k}$ 
given in Proposition \ref{representing} is far from minimal, thus seeming unnecessarily large. 
Alas also that presentation is useful, 
as it leads to a compact formula for the coordinate change between two charts 
$\Hi^{\Delta}_{S/k}$ and $\Hi^{\Pi}_{S/k}$ of $\Hi^{d}_{S/k}$.
We will carry this out in Section \ref{changingcharts} below. 

%%%%%%%%%%%%%%%%%%%%%%%%%%%%%%%%%%%%%%

\section{A smaller set of generators}\label{examples}

For illustrating the presentation of $\Hi^{\Delta}_{S/k}$ 
given in the last section, we go through a few examples. 
These will also serve as a motivation for Theorem \ref{fewer} below,
which is a substantial improvement of Proposition \ref{representing}. 
In all examples we only study $\Hi^{\Delta}_{S/k}$, and not $\Hi^{\prec\Delta}_{S/k}$, 
as the latter arises from the former by simply setting $T_{\alpha,\beta}=0$ for all $\alpha\prec\beta$. 

\begin{center}
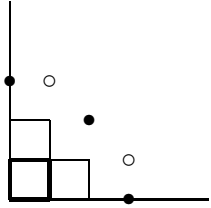
\begin{figure}[ht]
  \begin{picture}(135,115)
    % the rest of the coordinate axes
    \put(45,20){\line(1,0){60}}
    \put(30,35){\line(0,1){60}}
    % $\Delta$
    \linethickness{.5mm}
    \put(30,20){\line(1,0){15}}
    \put(30,35){\line(1,0){15}}
    \put(30,20){\line(0,1){15}}
    \put(45,20){\line(0,1){15}}
    \thinlines
    % $\Delta^{(1)}$
    \put(45,35){\line(1,0){15}}
    \put(30,50){\line(1,0){15}}
    \put(45,35){\line(0,1){15}}
    \put(60,20){\line(0,1){15}}
    % $\Delta^{(2)}$
    \put(30,65){\circle*{4}}
    \put(45,65){\circle{4}}
    \put(60,50){\circle*{4}}
    \put(75,35){\circle{4}}
    \put(75,20){\circle*{4}}
  \end{picture}
\caption{A standard set $\Delta$ together with $\Delta^{(1)}$ and $\Delta^{(2)}$.}
\label{figex1}
\end{figure}
\end{center}

\begin{ex}
  Consider the following standard set $\Delta$ and its borders:
  \begin{equation*}
    \begin{split}
      \Delta&=\{(0,0),(1,0),(0,1),(1,1)\}\,,\\
      \Delta^{(1)}&=\{(2,0),(0,2),(2,1),(1,2)\}\,,\\
      \Delta^{(2)}&=\{(3,0),(3,1),(2,2),(1,3),(0,4)\}\,.
    \end{split}
  \end{equation*}
  Figure \ref{figex1} shows $\Delta$, drawn in thick lines; $\Delta^{(1)}$, 
  drawn in thin lines; and $\Delta^{(2)}$, marked by circles. 
  
  In view of the presence of $I^\Delta_{1}$ in the ideal $I^\Delta$ of Proposition \ref{representing}, 
  we replace the polynomial ring $k[T_{\alpha,\beta};\alpha\in\Delta\cup\Delta^{(1)},\beta\in\Delta]$
  of that theorem by $R=k[T_{\alpha,\beta};\alpha\in\Delta^{(1)},\beta\in\Delta]$. 
  Then 
  \begin{equation*}
    \Hi^{\Delta}_{S/k}={\rm Spec}\,R/I^\Delta\,, 
  \end{equation*}
  where $I^\Delta$ is the sum of the three ideals
  \begin{equation*}
    \begin{split}
      I^\Delta_{2,1}&=(T_{(1,2),(0,0)}-T_{(0,2),(1,0)}T_{(2,0),(0,0)}-T_{(0,2),(1,1)}T_{(2,1),(0,0)},\\
      &T_{(1,2),(1,0)}-T_{(0,2),(1,0)}T_{(2,0),(1,0)}-T_{(0,2),(1,1)}T_{(2,1),(1,0)}-T_{(0,2),(0,0)},\\
      &T_{(1,2),(0,1)}-T_{(0,2),(1,0)}T_{(2,0),(0,1)}-T_{(0,2),(1,1)}T_{(2,1),(0,1)},\\
      &T_{(1,2),(1,1)}-T_{(0,2),(1,0)}T_{(2,0),(1,1)}-T_{(0,2),(1,1)}T_{(2,1),(1,1)}-T_{(0,2),(0,1)})\,,
    \end{split}
  \end{equation*}
  \begin{equation*}
    \begin{split}
      I^\Delta_{2,2}&=(T_{(2,1),(0,0)}-T_{(2,0),(0,1)}T_{(0,2),(0,0)}-T_{(2,0),(1,1)}T_{(1,2),(0,0)},\\
      &T_{(2,1),(1,0)}-T_{(2,0),(0,1)}T_{(0,2),(1,0)}-T_{(2,0),(1,1)}T_{(1,2),(1,0)},\\
      &T_{(2,1),(0,1)}-T_{(2,0),(0,1)}T_{(0,2),(0,1)}-T_{(2,0),(1,1)}T_{(1,2),(0,1)}-T_{(2,0),(0,0)},\\
      &T_{(2,1),(1,1)}-T_{(2,0),(0,1)}T_{(0,2),(1,1)}-T_{(2,0),(1,1)}T_{(1,2),(1,1)}-T_{(2,0),(1,0)})
    \end{split}
  \end{equation*}
  and
  \begin{equation*}
    \begin{split}
      I^\Delta_{3}&=(T_{(1,2),(1,0)}T_{(2,0),(0,0)}+T_{(1,2),(1,1)}T_{(2,1),(0,0)}\\
      &-T_{(2,1),(0,1)}T_{(0,2),(0,0)}-T_{(2,1),(1,1)}T_{(1,2),(0,0)},\\
      &T_{(1,2),(1,0)}T_{(2,0),(1,0)}+T_{(1,2),(1,1)}T_{(2,1),(1,0)}+T_{(1,2),(0,0)}\\
      &-T_{(2,1),(0,1)}T_{(0,2),(1,0)}-T_{(2,1),(1,1)}T_{(1,2),(1,0)},\\
      &T_{(1,2),(1,0)}T_{(2,0),(0,1)}+T_{(1,2),(1,1)}T_{(2,1),(0,1)}\\
      &-T_{(2,1),(0,1)}T_{(0,2),(0,1)}-T_{(2,1),(1,1)}T_{(1,2),(0,1)}-T_{(2,1),(0,0)},\\
      &T_{(0,2),(1,0)}T_{(2,0),(1,1)}+T_{(0,2),(1,1)}T_{(2,1),(1,1)}+T_{(0,2),(0,1)}\\
      &-T_{(2,1),(0,1)}T_{(0,2),(1,1)}-T_{(2,1),(1,1)}T_{(1,2),(1,1)}-T_{(2,1),(1,0)})\,.
    \end{split}
  \end{equation*}
  The ideals $I^\Delta_{2,1}$ and $I^\Delta_{2,2}$ 
  correspond to the equalities derived in Step 2 of the proof of Proposition \ref{representing}, 
  i.e. from the multiplicativity condition $\phi(x^{\alpha+\lambda})=\phi(x^\alpha)\phi(x^\lambda)$
  for $\alpha,\alpha+\lambda\in\Delta^{(1)}$. 
  As for $I^\Delta_{2,1}$, we choose $\alpha=(0,2),\lambda=(1,0)$. 
  As for $I^\Delta_{2,2}$, we choose $\alpha=(2,0),\lambda=(0,1)$. 
  
  The ideal $I^\Delta_{3}$
  corresponds to the equalities derived in Step 3 of the proof of Proposition \ref{representing}, 
  i.e. from the unambiguity condition $\phi(x^{\alpha+\lambda})=\phi(x^{\alpha^\prime+\lambda^\prime})$
  for $\alpha,\alpha^\prime\in\Delta^{(1)}$ such that 
  $\alpha+\lambda=\alpha^\prime+\lambda^\prime\in\Delta^{(2)}$. 
  The only choice for that is $\alpha=(1,2)$, $\lambda=(1,0)$, 
  $\alpha ^\prime =(2,1)$,$\lambda^\prime=(0,1)$. 
\end{ex}

In other words, unambiguity only has to be guaranteed at $(2,2)\in\Delta^{(2)}$. 
That point a corner of $\Delta\cup\Delta^{(1)}$, as we see from Figure \ref{figex2}. 
Moreover, it arises from the only edge point of $\Delta$ by addition of $(1,1)$. 

\begin{center}
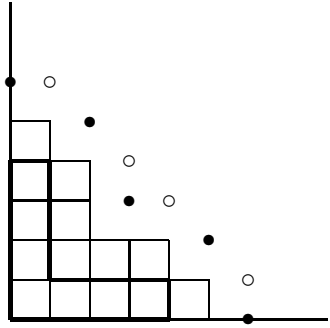
\begin{figure}[ht]
  \begin{picture}(180,160)
    % the rest of the coordinate system
    \put(90,20){\line(1,0){60}}
    \put(30,80){\line(0,1){60}}
    % $\Delta$
    \linethickness{.5mm}
    \put(30,20){\line(1,0){60}}
    \put(30,20){\line(0,1){60}}
    \put(45,35){\line(1,0){45}}
    \put(45,35){\line(0,1){45}}
    \put(30,80){\line(1,0){15}}
    \put(90,20){\line(0,1){15}}
    \thinlines
    % $\Delta^{(1)}$
    \put(30,35){\line(1,0){15}}
    \put(30,50){\line(1,0){15}}
    \put(30,65){\line(1,0){15}}
    \put(30,95){\line(1,0){15}}
    \put(45,20){\line(0,1){15}}
    \put(60,20){\line(0,1){15}}
    \put(75,20){\line(0,1){15}}
    \put(105,20){\line(0,1){15}}
    \put(45,50){\line(1,0){45}}
    \put(45,65){\line(1,0){15}}
    \put(45,80){\line(1,0){15}}
    \put(90,35){\line(1,0){15}}
    \put(60,35){\line(0,1){45}}
    \put(75,35){\line(0,1){15}}
    \put(90,35){\line(0,1){15}}
    \put(45,80){\line(0,1){15}}
    % $\Delta^{(2)}$
    \put(30,110){\circle*{4}}
    \put(45,110){\circle{4}}
    \put(60,95){\circle*{4}}
    \put(75,80){\circle{4}}
    \put(75,65){\circle*{4}}
    \put(90,65){\circle{4}}
    \put(105,50){\circle*{4}}
    \put(120,35){\circle{4}}
    \put(120,20){\circle*{4}}
  \end{picture}
\caption{A standard set $\Delta$ together with $\Delta^{(1)}$ and $\Delta^{(2)}$.}
\label{figex2}
\end{figure}
\end{center}

\begin{ex}\label{ex2}
  Consider the standard set $\Delta$, along with its borders $\Delta^{(1)}$ and $\Delta^{(2)}$, 
  as depicted in Figure \ref{figex2}. 
  We define the polynomial ring $R$ by the same formula as in the previous example. Then
  \begin{equation*}
    \Hi^{\Delta}_{S/k}={\rm Spec}\,R/I^\Delta\,, 
  \end{equation*}
  where 
  \begin{equation*}
    I^\Delta=I^\Delta_{2,1}+\ldots+I_{2,6}+I^\Delta_{3,1}+I_{3,2}\,.
  \end{equation*}
  We do not write down the summands of $I$ explicitly, but rather describe them as follows: 
  The ideals $I^\Delta_{2,i}$ and $I^\Delta_{3,j}$ correspond to the equalities derived in Step 2 and Step 3, resp., 
  of the proof of Proposition \ref{representing}, according to the following values of $\alpha,\lambda$, 
  and $\alpha^\prime,\lambda^\prime$, resp.
  \begin{center}
    \begin{tabular}{|c|c|c|}
      \hline
      $I^\Delta_{2,1}$ & $\alpha=(0,5)$ & $\lambda=(1,0)$ \\ \hline
      $I^\Delta_{2,2}$ & $\alpha=(2,3)$ & $\lambda=(0,1)$ \\ \hline
      $I^\Delta_{2,3}$ & $\alpha=(2,2)$ & $\lambda=(0,1)$ \\ \hline
      $I^\Delta_{2,4}$ & $\alpha=(2,2)$ & $\lambda=(1,0)$ \\ \hline
      $I^\Delta_{2,5}$ & $\alpha=(3,2)$ & $\lambda=(1,0)$ \\ \hline
      $I^\Delta_{2,6}$ & $\alpha=(5,0)$ & $\lambda=(0,1)$ \\ \hline
    \end{tabular}
  \end{center}
  \begin{center}
    \begin{tabular}{|c|c|c|c|c|}
      \hline
      $I^\Delta_{3,1}$ & $\alpha=(1,5)$ & $\lambda=(1,0)$ & $\alpha^\prime=(2,4)$ & $\lambda^\prime=(0,1)$ \\ \hline
      $I^\Delta_{3,2}$ & $\alpha=(4,2)$ & $\lambda=(0,1)$  & $\alpha^\prime=(5,1)$ & $\lambda^\prime=(0,1)$\\ \hline
    \end{tabular}
  \end{center}
\end{ex}

The interesting observation here is that the summands $I_{3,1}$ and $I_{3,2}$ 
correspond to the two edge points $(1,4)$ and $(4,1)$ of $\Delta$. 

\begin{center}
\begin{figure}
  \begin{picture}(120,175)
    % the rest of the coordinate system
    \put(55,50){\line(1,0){45}}
    \put(40,110){\line(0,1){45}}
    \put(29.8,43.2){\line(-3,-2){30}}
    \linethickness{.5mm}
    % $\Delta$
    \put(40,50){\line(1,0){15}}
    \put(40,110){\line(1,0){15}}
    \put(29.8,43.2){\line(1,0){15}}
    \put(29.8,103.2){\line(1,0){15}}
    \put(40,50){\line(0,1){60}}
    \put(55,50){\line(0,1){60}}
    \put(29.8,43.2){\line(0,1){60}}
    \put(44.8,43.2){\line(0,1){60}}
    \put(40,50){\line(-3,-2){10}}
    \put(40,50.1){\line(-3,-2){10}}
    \put(40,50.2){\line(-3,-2){10}}
    \put(40,50.3){\line(-3,-2){10}}
    \put(40,50.4){\line(-3,-2){10}}
    \put(40,49.9){\line(-3,-2){10}}
    \put(40,49.8){\line(-3,-2){10}}
    \put(40,49.7){\line(-3,-2){10}}
    \put(40,49.6){\line(-3,-2){10}}
    \put(40,49.5){\line(-3,-2){10}}
    \put(40,49.4){\line(-3,-2){10}}
    \put(55,50){\line(-3,-2){10}}
    \put(55,50.1){\line(-3,-2){10}}
    \put(55,50.2){\line(-3,-2){10}}
    \put(55,50.3){\line(-3,-2){10}}
    \put(55,50.4){\line(-3,-2){10}}
    \put(55,49.9){\line(-3,-2){10}}
    \put(55,49.8){\line(-3,-2){10}}
    \put(55,49.7){\line(-3,-2){10}}
    \put(55,49.6){\line(-3,-2){10}}
    \put(55,49.5){\line(-3,-2){10}}
    \put(55,49.4){\line(-3,-2){10}}
    \put(40,110){\line(-3,-2){10}}
    \put(40,110.1){\line(-3,-2){10}}
    \put(40,110.2){\line(-3,-2){10}}
    \put(40,110.3){\line(-3,-2){10}}
    \put(40,110.4){\line(-3,-2){10}}
    \put(40,109.9){\line(-3,-2){10}}
    \put(40,109.8){\line(-3,-2){10}}
    \put(40,109.7){\line(-3,-2){10}}
    \put(40,109.6){\line(-3,-2){10}}
    \put(40,109.5){\line(-3,-2){10}}
    \put(40,109.4){\line(-3,-2){10}}
    \put(55,110){\line(-3,-2){10}}
    \put(55,110.1){\line(-3,-2){10}}
    \put(55,110.2){\line(-3,-2){10}}
    \put(55,110.3){\line(-3,-2){10}}
    \put(55,110.4){\line(-3,-2){10}}
    \put(55,109.9){\line(-3,-2){10}}
    \put(55,109.8){\line(-3,-2){10}}
    \put(55,109.7){\line(-3,-2){10}}
    \put(55,109.6){\line(-3,-2){10}}
    \put(55,109.5){\line(-3,-2){10}}
    \put(55,109.4){\line(-3,-2){10}}
    \thinlines
    % the part of $\Delta^{(1)}$ toward $(1,0)$
    \put(44.8,43.2){\line(1,0){15}}
    \put(29.8,58.2){\line(1,0){30}}
    \put(29.8,73.2){\line(1,0){30}}
    \put(29.8,88.2){\line(1,0){30}}
    \put(44.8,103.2){\line(1,0){15}}
    \put(40,65){\line(1,0){30}}
    \put(40,80){\line(1,0){30}}
    \put(40,95){\line(1,0){30}}
    \put(55,110){\line(1,0){15}}
    \put(59.8,43.2){\line(0,1){60}}
    \put(70,50){\line(0,1){60}}
    \put(70,50){\line(-3,-2){10}}
    \put(70,65){\line(-3,-2){10}}
    \put(70,80){\line(-3,-2){10}}
    \put(70,95){\line(-3,-2){10}}
    \put(70,110){\line(-3,-2){10}}
    % the part of $\Delta^{(1)}$ toward $(0,1)$
    \put(40,125){\line(1,0){15}}
    \put(29.8,118.2){\line(1,0){15}}
    \put(40,110){\line(0,1){15}}
    \put(55,110){\line(0,1){15}}
    \put(29.8,103.2){\line(0,1){15}}
    \put(44.8,103.2){\line(0,1){15}}
    \put(40,125){\line(-3,-2){10}}
    \put(55,125){\line(-3,-2){10}}
    % the part of $\Delta^{(1)}$ toward $(-3,-2)$
    \put(19.8,36.6){\line(1,0){15}}
    \put(19.8,51.6){\line(1,0){15}}
    \put(19.8,66.6){\line(1,0){15}}
    \put(19.8,81.6){\line(1,0){15}}
    \put(19.8,96.6){\line(1,0){15}}
    \put(19.8,36.6){\line(0,1){60}}
    \put(34.8,36.6){\line(0,1){60}}
    \put(40,65){\line(-3,-2){20}}
    \put(40,80){\line(-3,-2){20}}
    \put(40,95){\line(-3,-2){20}}
    \put(55,65){\line(-3,-2){20}}
    \put(55,80){\line(-3,-2){20}}
    \put(55,95){\line(-3,-2){20}}
    \put(44.8,43.2){\line(-3,-2){10}}
    \put(44.8,103.2){\line(-3,-2){10}}
    \put(29.8,103.2){\line(-3,-2){10}}
    % some corners
    \put(49.8,36.6){\circle*{4}}
    \put(19.8,111.6){\circle*{4}}
    \put(70,125){\circle*{4}}
  \end{picture}
\caption{A standard set $\Delta$ together with $\Delta^{(1)}$ and three elements of $\Delta^{(2)}$.}
\label{figex3}
\end{figure}
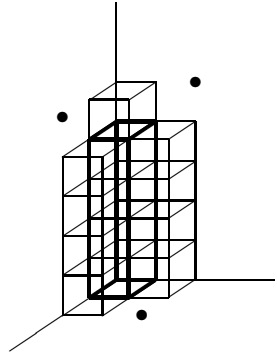
\end{center}

\begin{ex}\label{ex3}
  Consider the standard set $\Delta$, along with its borders $\Delta^{(1)}$ and $\Delta^{(2)}$, 
  as depicted in Figure \ref{figex3}. 
  We define the polynomial ring $R$ by the same formula as in the previous two examples. Then
  \begin{equation*}
    \Hi^{\Delta}_{S/k}={\rm Spec}\,R/I^\Delta\,, 
  \end{equation*}
  where 
  \begin{equation*}
    I^\Delta=I^\Delta_{2,1}+\ldots+I^\Delta_{2,30}+I^\Delta_{3,1}+I^\Delta_{3,2}+I^\Delta_{3,3}\,.
  \end{equation*}
  The summands of $I^\Delta$ have analogous descriptions as in the previous example. 
  The ideals $I^\Delta_{2,i}$ correspond to the 30 possible values of $\alpha,\lambda$ 
  such that $\alpha$ and $\alpha+\lambda$ both lie in $\Delta^{(1)}$. 
  The ideals $I^\Delta_{3,i}$ correspond to the following values of $\alpha,\lambda$ 
  and $\alpha^\prime,\lambda^\prime$, resp.
  \begin{center}
    \begin{tabular}{|c|c|c|c|c|}
      \hline
      $I^\Delta_{3,1}$ & $\alpha=(1,2,0)$ & $\lambda=(1,0,0)$ & 
      $\alpha^\prime=(2,1,0)$ & $\lambda^\prime=(0,1,0)$ \\ \hline
      $I^\Delta_{3,2}$ & $\alpha=(1,0,5)$ & $\lambda=(1,0,0)$  
      & $\alpha^\prime=(2,0,4)$ & $\lambda^\prime=(0,0,1)$\\ \hline
      $I^\Delta_{3,3}$ & $\alpha=(0,4,2)$ & $\lambda=(0,1,0)$  
      & $\alpha^\prime=(0,5,1)$ & $\lambda^\prime=(0,0,1)$\\ \hline
    \end{tabular}
  \end{center}
\end{ex}

Again the summands $I^\Delta_{3,1}$, $I^\Delta_{3,2}$ and $I^\Delta_{3,3}$ 
correspond to the three edge points $(1,1,0)$, $(1,0,4)$ and $(0,4,1)$ of $\Delta$. 
The following theorem explains this finding. 

\begin{thm}\label{fewer}
  Let $R^\Delta=R/I^\Delta$ be the coordinate ring of $\Hi^{\Delta}_{S/k}$ as presented in Proposition \ref{representing}, 
  where $N=\Delta$. Then the two summands $I^\Delta_{2}$ and $I^\Delta_{3}$, resp., of $I^\Delta$ can be replaced by the ideals
  \begin{equation*}
    \begin{split}
      I^\Delta_{2,e}&=(T_{\alpha+\lambda,\beta}-\sum_{\gamma\in\Delta}T_{\alpha,\gamma}T_{\gamma+\lambda,\beta};\\
      &\alpha\in\Delta^{(1)},\lambda\in\{e_{1},\ldots,e_{n}\}
      \text{ s.t. }\alpha+\lambda\in\Delta^{(1)},\beta\in\Delta)\text{ and}\\
      I^\Delta_{3,e}&=(\sum_{\gamma\in\Delta}T_{\epsilon+\lambda^\prime,\gamma}T_{\gamma+\lambda,\beta}
      -\sum_{\gamma\in\Delta}T_{\epsilon+\lambda,\gamma}T_{\gamma+\lambda^\prime,\beta};\\
      &(\epsilon,\lambda,\lambda^\prime)\text{ is an edge triple of }\Delta,\beta\in\Delta)\,,
    \end{split}
  \end{equation*}
  resp., in the polynomial ring $R$. 
\end{thm}

\begin{proof}
  The statement about $I^\Delta_{2}$ is easy to prove: If $\alpha\in\Delta$, 
  then $T_{\alpha,\gamma}=\delta_{\alpha,\gamma}$, hence 
  $T_{\alpha+\lambda,\beta}-\sum_{\gamma\in\Delta}T_{\alpha,\gamma}T_{\gamma+\lambda,\beta}=0$. 
  This polynomial can therefore be eliminated from the set of generators of $I^\Delta_{2}$. 
  In the rest of the proof, we show that we may replace $I^\Delta_{3}$ by $I^\Delta_{3,e}$.
  
  We define $\Gamma_{0}$ to be the set of all sums 
  $\alpha+\lambda=\alpha^\prime+\lambda^\prime\in\Delta^{(2)}$
  where $\alpha\neq\alpha^\prime\in\Delta^{(1)}$ and $\lambda\neq\lambda^\prime\in\{e_{1},\ldots,e_{n}\}$. 
  Remember that the generators of the ideal $I^\Delta_{3}$ correspond to equations
  \begin{equation}\label{nonandprime}
    \phi(x^\alpha)\phi(x^\lambda)=\phi(x^{\alpha^\prime})\phi(x^{\lambda^\prime})\,,
  \end{equation}
  where $\alpha+\lambda=\alpha^\prime+\lambda^\prime\in\Gamma_{0}$, in the following way: 
  Equation \eqref{nonandprime} translates into the system of equations 
  \begin{equation*}
    \forall\beta\in\Delta:\sum_{\gamma\in\Delta}a_{\alpha,\gamma}a_{\gamma+\lambda,\beta}
    =a_{\alpha^\prime,\gamma}a_{\gamma+\lambda^\prime,\beta}\,,
  \end{equation*}
  which system is then replaced by the generators 
  \begin{equation*}
    \sum_{\gamma\in\Delta}T_{\alpha,\gamma}T_{\gamma+\lambda,\beta}
    -\sum_{\gamma\in\Delta}T_{\alpha^\prime,\gamma}T_{\gamma+\lambda^\prime,\beta}\,,\beta\in\Delta
  \end{equation*}
  of $I^\Delta_{3}$. 
  Analogously, the generators of the ideal $I^\Delta_{2,e}$ correspond to equations 
  \begin{equation}\label{2}
    \phi(x^{\alpha+\lambda})=\phi(x^\alpha)\phi(x^\lambda)\,,
  \end{equation}
  where $\alpha,\alpha+\lambda\in\Delta^{(1)}$. 
  Therefore we may assume that \eqref{2} holds for all $\alpha,\alpha+\lambda\in\Delta^{(1)}$. 
  (As we have seen in the first paragraph of the present proof, 
  the case where $\alpha$ lies in $\Delta$ gives a trivial generator of $I^\Delta_{2}$. 
  Accordingly, equation \eqref{2} is trivial if $\alpha\in\Delta$.)
  
  Our first claim is that if \eqref{nonandprime} holds for all 
  $\alpha+\lambda=\alpha^\prime+\lambda^\prime\in\Gamma_{0}\cap\mathscr{C}(\Delta\cup\Delta^{(1)})$, 
  then \eqref{nonandprime} automatically holds for all 
  $\alpha+\lambda=\alpha^\prime+\lambda^\prime\in\Gamma_{0}$.   
  For this we define $\Gamma=\Gamma_{0}-\mathscr{C}(\Delta\cup\Delta^{(1)})$ 
  and prove the claim by induction over $\Gamma$. 
  Take an arbitrary $\alpha+\lambda=\alpha^\prime+\lambda^\prime\in\Gamma$. 
  First we observe that there exists some $\nu\in\{e_{1},\ldots,e_{n}\}$ 
  such that $\alpha+\lambda-\nu$ lies in $\Delta^{(2)}$, 
  since if all $\alpha+\lambda-\nu$ were to lie either in $\Delta^{(1)}$ or outside $\mathbb{N}^n$, 
  then $\alpha+\lambda$ would lie in $\mathscr{C}(\Delta\cup\Delta^{(1)})$, a contradiction. 
  As $\alpha+\lambda-\nu=\alpha^\prime+\lambda^\prime-\nu$ lies in $\Delta^{(2)}$ and not in $\Delta^{(1)}$, 
  the element $\nu$ equals neither $\lambda$ nor $\lambda^\prime$. 
  Therefore both $\alpha-\nu$ and $\alpha^\prime-\nu$ lie in $\mathbb{N}^n$, 
  hence in $\Delta^{(1)}$.
  It follows that $(\alpha-\nu)+\lambda=(\alpha^\prime-\nu)+\lambda^\prime$ lies in $\Gamma_{0}$. 
  If $(\alpha-\nu)+\lambda$ lies in $\Gamma_{0}\cap\mathscr{C}(\Delta\cup\Delta^{(1)})$, 
  then \eqref{nonandprime} holds for $(\alpha-\nu)+\lambda=(\alpha^\prime-\nu)+\lambda^\prime$
  by assumption. 
  If $(\alpha-\nu)+\lambda$ lies in the complement, 
  then \eqref{nonandprime} holds for $(\alpha-\nu)+\lambda=(\alpha^\prime-\nu)+\lambda^\prime$
  by our induction hypothesis, as $(\alpha-\nu)+\lambda\prec\alpha+\lambda$. 
  In both cases we obtain 
  \begin{equation*}
    \phi(x^\alpha)\phi(x^\lambda)
    =\phi(x^{\alpha-\nu})\phi(x^\nu)\phi(x^\lambda)
    =\phi(x^{\alpha^\prime-\nu})\phi(x^\nu)\phi(x^{\lambda^\prime})
    =\phi(x^{\alpha^\prime})\phi(x^{\lambda^\prime})\,.
  \end{equation*}
  Here we used \eqref{2} for the first and the last equality. 
  Therefore \eqref{nonandprime} holds for $\alpha+\lambda=\alpha^\prime+\lambda^\prime$, 
  and the first claim is proved. 
  
  Our second claim is that if \eqref{nonandprime} holds for all
  $\alpha+\lambda=\alpha^\prime+\lambda^\prime\in\Gamma_{0}\cap\mathscr{C}(\Delta\cup\Delta^{(1)})$ 
  such that $\alpha-\lambda^\prime=\alpha^\prime-\lambda$ lies in $\Delta$, 
  then \eqref{nonandprime} holds for all
  $\alpha+\lambda=\alpha^\prime+\lambda^\prime\in\Gamma_{0}\cap\mathscr{C}(\Delta\cup\Delta^{(1)})$ 
  with no additional restriction. 
  Indeed, from $\lambda\neq\lambda^\prime$ it follows that $\alpha-\lambda^\prime=\alpha^\prime-\lambda$
  lies in $\mathbb{N}^n$, and therefore either in $\Delta^{(1)}$ or in $\Delta$. 
  In the former case we compute
  \begin{equation*}
    \phi(x^\alpha)\phi(x^\lambda)
    =\phi(x^{\alpha-\lambda^\prime})\phi(x^{\lambda^\prime})\phi(x^\lambda)
    =\phi(x^{\alpha^\prime-\lambda})\phi(x^{\lambda^\prime})\phi(x^\lambda)
    =\phi(x^{\alpha^\prime})\phi(x^{\lambda^\prime})\,,
  \end{equation*}
  again using \eqref{2} for the first and the last equality. 
  In the latter case, there is nothing to prove, as \eqref{nonandprime} holds by assumption. 
  
  Each of the remaining $\alpha+\lambda=\alpha^\prime+\lambda^\prime$
  lies in $\Gamma_{0}\cap\mathscr{C}(\Delta\cup\Delta^{(1)})$ and has the property that 
  the element $\epsilon=\alpha-\lambda^\prime=\alpha^\prime-\lambda$ lies in $\Delta$. 
  This is equivalent to $(\epsilon,\lambda,\lambda^\prime)$ forming an edge triple. 
  The theorem follows.
\end{proof}

Note that though $I^\Delta_{1}+I^\Delta_{2}+I^\Delta_{3}=I^\Delta_{1}+I^\Delta_{2,e}+I^\Delta_{3,e}$, 
the ideal $I^\Delta_{2,e}$ is strictly smaller than $I^\Delta_{2}$, 
and the ideal $I^\Delta_{3,e}$ is strictly smaller than $I^\Delta_{3}$. 
Of course the presentations of the coordinate ring of $\Hi^{\prec\Delta}_{S/k}$ 
given in Corollaries \ref{prime} and \ref{min} can be reformulated in the spirit of Theorem \ref{fewer}. 

Example \ref{ex2} illustrates the third claim in the proof, and Example \ref{ex3} illustrates the first claim. 
The theorem we just proved is a substantial improvement of Proposition \ref{representing}, 
as it makes the number of generators needed for $I^\Delta$ much smaller. 
The concept of {\it across-the-corner neighbors} of \cite{krarticle} 
uses the same idea which we used in the proof of the second claim here. 
Yet our set of generators for $I^\Delta$ of Theorem \ref{fewer} is smaller than the set of generators of the cited article. 

\begin{ex}\label{triangleagain}
  Let $k=\mathbb{Z}$ and $\Delta=\{0,e_1,e_2\}$ be as in Example \ref{triangle}. 
  Then $\delta^{(1)}=\{2e_1,e_1+e_2,2e_2\}$. 
  We use the variables $T_{\alpha,\beta}$, $\alpha\in\Delta^{(1)}$, 
  $\beta\in\Delta$ for presenting $R^\Delta$ as in the theorem. The ideal $I^\Delta_{2,e}$ vanishes, 
  as there are no $\alpha\in\Delta$, $\lambda\in\{e_1,e_2\}$ such that $\alpha+\lambda\in\Delta^{(1)}$. 
  There are two edge points, $e_1$ and $e_2$, hence six generators of $I^\Delta_{3,e}$. 
  They boil down to the following five conditions on $T_{\alpha,\beta}$:
  \begin{equation*}
    \begin{split}
      T_{(2,0),(0,0)}&=T_{(1,1),(1,0)}T_{(2,0),(0,1)}+T_{(1,1),(0,1)}^2\\
      &-T_{(2,0),(1,0)}T_{(1,1),(0,1)}-T_{(2,0),(0,1)}T_{(0,2),(0,1)}\,,\\
      T_{(1,1),(0,0)}&=T_{(2,0),(0,1)}T_{(0,2),(1,0)}-T_{(1,1),(0,1)}T_{(1,1),(1,0)}\,,\\
      T_{(0,2),(0,0)}&=T_{(1,1),(1,0)}^2+T_{(1,1),(0,1)}T_{(0,2),(1,0)}\\
      &-T_{(0,2),(1,0)}T_{(2,0),(1,0)}-T_{(0,2),(0,1)}T_{(1,1),(1,0)}
    \end{split}
  \end{equation*}
  and 
  \begin{equation*}
    \begin{split}
      &T_{(2,0),(1,0)}T_{(1,1),(0,0)}+T_{(2,0),(0,1)}T_{(0,2),(0,0)}\\
      -&T_{(1,1),(1,0)}T_{(2,0),(0,0)}-T_{(1,1),(0,1)}T_{(1,1),(0,0)}=0\,,\\
      &T_{(1,1),(1,0)}T_{(1,1),(0,0)}+T_{(1,1),(0,1)}T_{(0,2),(0,0)}\\
      -&T_{(0,2),(1,0)}T_{(2,0),(0,0)}-T_{(0,2),(0,1)}T_{(1,1),(0,0)}=0\,.
    \end{split}
  \end{equation*}
  Upon substituting $T_{(2,0),(0,0)}$, $T_{(1,1),(0,0)}$ and $T_{(0,2),(0,0)}$ into the last two equations, 
  we see that the last two equations are trivial. 
  Therefore $\Hi^\Delta_{S/k}$ is the 6-dimensional affine space with parameters 
  $T_{(2,0),(1,0)}$, $T_{(2,0),(0,1)}$, $T_{(1,1),(1,0)}$, $T_{(1,1),(0,1)}$, $T_{(0,2),(1,0)}$, $T_{(0,2),(0,1)}$. 
  (Note that this also follows from Huibregtse's result on ``sawtooth'' standard sets, 
  which we discussed at the end of Section \ref{affines} above.)
  In Example \ref{triangle}, we chose 
  $T_{(2,0),(1,0)}=T_{(2,0),(0,1)}=1$, $T_{(1,1),(1,0)}=T_{(1,1),(0,1)}=2$, $T_{(0,2),(1,0)}=T_{(0,2),(0,1)}=1$
  for obtaining $I_a$. 
  This makes $I_a$ monic for the graded lexicographic order, but not monic for the lexicographic order where $x_1\succ x_2$. 
  For obtaining $I_b$, we replaced $T_{(0,2),(1,0)}=1$ by $T_{(0,2),(1,0)}=0$. 
  This makes this ideal $I_b$ monic for both the graded lexicographic and the lexicographic order. 
\end{ex}

\begin{ex}\label{axis}
  For $\Delta=\{0,\ldots,(d-1)e_{1}\}\subset\mathbb{N}^n$, 
  the scheme $\Hi^\Delta_{S/k}$ is the $dn$-dimensional affine space with coordinates 
  $T_{de_{1},\beta},T_{e_{2},\beta},\ldots,T_{e_{n},\beta}$, for $\beta\in\Delta$. 
\end{ex}

We just give a hint for the proof of this: First consider the case $n=2$. For $a=0,\ldots,d-2$, 
the equations $\phi(x^{(a+1,1)})=\phi(x^{(a,1)})\phi(x^{(1,0)})$ give explicit formulas for all 
$T_{(a+1,1),\beta}$ as polynomials in $T_{(0,1),\gamma}$ and $T_{(d,0),\gamma}$, $\gamma\in\Delta$. 
In the system of polynomial equations corresponding to the equation 
$\phi(x^{(d-1,1)})\phi(x^{(1,0)})=\phi(x^{(d,0)})\phi(x^{(0,1)})$, 
replace each $T_{(a+1,1),\beta}$, for $a=0,\ldots,d-2$, by the polynomial expression from above. 
Then it turns out that the above system of equations is trivial. The assertion follows in the case $n=2$. 
For larger $n$, fix an $i\in\{2,\ldots,n\}$ and derive analogous formulas for 
$T_{(a+1)e_{i},1)\beta}$ as polynomials in $T_{e_{i},\gamma}$ and $T_{(d,0),\gamma}$, $\gamma\in\Delta$.
The system of polynomial equations corresponding to the equation
$\phi(x^{(d-1)e_{1}+e_{i})})\phi(x^{e_{1}})=\phi(x^{de_{1}})\phi(x^{e_{i}})$ is trivial again. 
By Theorem \ref{fewer}, these equations, for $i=2,\ldots,n$, are all we have to study. 
The assertion follows for all $n$. 

%%%%%%%%%%%%%%%%%%%%%%%%%%%%%%%%%%%%%%

\section{The universal objects}\label{universal}

Equation \eqref{ftomatrix} describes the transition between the matrix 
$(a_{\alpha\gamma})$ of a homomorphism $\phi\in\HHi^{\prec\Delta}_{S/k}(B)$ 
and the elements $f_{\alpha}$ of the kernel of $\phi$.
Together with Corollary \ref{prime} and Proposition \ref{representing}, 
this enables us to directly write down the universal objects. 

\begin{pro}\label{prouniversal}
  \begin{enumerate}
    \item[(i)] Let $R^\Delta$ be the coordinate ring of $\Hi^{\Delta}_{S/k}$ 
      as presented in Proposition \ref{representing} or Theorem \ref{fewer}. 
      Then the universal object of the representable functor $\HHi^{\Delta}_{S/k}$
      is the affine scheme 
      \begin{equation*}
        U^\Delta={\rm Spec}\,R^\Delta[x]/
        (x^\alpha-\sum_{\beta\in\Delta}T_{\alpha,\beta}x^\beta;\alpha\in N\cup N^{(1)})
      \end{equation*}
      over $\Hi^{\Delta}_{S/k}={\rm Spec}\,R^\Delta$. 
    \item[(ii)] Let $R^{\prec\Delta}$ be the coordinate ring of $\Hi^{\prec\Delta}_{S/k}$ 
      as presented in Corollaries \ref{prime} or \ref{min}. 
      Then the universal object of the representable functor $\HHi^{\prec\Delta}_{S/k}$
      is the affine scheme 
      \begin{equation*}
        U^{\prec\Delta}={\rm Spec}\,R^{\prec\Delta}[x]/(x^\alpha-\sum_{\beta\in\Delta,\,\beta\prec\alpha}
        T_{\alpha,\beta}x^\beta;\alpha\in\mathscr{C}(\Delta))
      \end{equation*}
      over $\Hi^{\prec\Delta}_{S/k}={\rm Spec}\,R^{\prec\Delta}$. 
    \end{enumerate}
\end{pro}

\begin{proof}
  (i) is clear. As for (ii), the only thing we have to prove is that
  for generating the ideal
  \begin{equation*}
    (x^\alpha-\sum_{\beta\in\Delta,\beta\prec\alpha}T_{\alpha,\,\beta}x^\beta;\alpha\in N\cup N^{(1)})\,,
  \end{equation*}
  it suffices take all $\alpha\in\mathscr{C}(\Delta)$. 
  However, in $R^{\prec\Delta}$ we have $T_{\alpha,\beta}=0$ whenever $\alpha\prec\beta$. 
  Therefore, the assertion follows from Lemma \ref{favouritelemma} (i). 
\end{proof}

Note that the theorem gives us 
\begin{equation*}
  U^\Delta\hookrightarrow\mathbb{A}^n_{\Hi^{\Delta}_{S/k}}\,,
  U^{\prec\Delta}\hookrightarrow\mathbb{A}^n_{\Hi^{\prec\Delta}_{S/k}}\,,
\end{equation*}
resp., as closed subschemes. 
Moreover, the coordinate rings of $U^\Delta$ and $U^{\prec\Delta}$, resp.,
are free over $R^\Delta$ and $R^{\prec\Delta}$, resp., by definition of the functor $\HHi^{\Delta}_{S/k}$.
In particular, the morphisms $U^\Delta\to\Hi^{\Delta}_{S/k}$ and $U^{\prec\Delta}\to\Hi^{\prec\Delta}_{S/k}$ 
are automatically flat. 

Proposition \ref{prouniversal} makes the statement precise that 
{\it the scheme $\Hi^{\prec\Delta}_{S/k}$ 
is the parametrizing space of all reduced Gr\"obner bases in $S$ with standard sets $\Delta$}:
A point ${\rm Spec}\,B\to\Hi^{\prec\Delta}_{S/k}$ 
is a homomorphism $R^{\prec\Delta}\to B$. 
In other words, we assign to the variables $T_{\alpha,\beta}$ values $a_{\alpha,\beta}\in B$ 
which satisfy the structural equations \eqref{structural}. 
Then we define
\begin{equation*}
  f_{\alpha}=x^\alpha-\sum_{\beta\in\Delta,\,\beta\prec\alpha}a_{\alpha,\beta}x^\beta\,.
\end{equation*}
Geometrically this means that we consider the cartesian diagram
\begin{equation*}
  \begin{CD}
    {\rm Spec}\,B[x]/(f_{\alpha};\alpha\in\mathscr{C}(\Delta))@>>>U^{\prec\Delta}\\
    @VVV @VVV\\
    {\rm Spec}\,B@>>>\Hi^{\prec\Delta}_{S/k}\,.
  \end{CD}
\end{equation*}
The structural equations guarantee that the polynomials $f_{\alpha}$ are a reduced Gr\"obner basis.
Equivalently, by Buchberger's $S$-pair criterion (see \cite{cox}, Section 2, \S6), 
the $S$-pairs of the various $f_{\alpha}$ reduce to zero modulo all $f_{\alpha}$.

%%%%%%%%%%%%%%%%%%%%%%%%%%%%%%%%%%%%%%

\section{Connectedness}\label{homogeneous}

\begin{pro}\label{primeconnected}
  There exists a morphism 
  \begin{equation*}
    g:\mathbb{A}^1_{k}\times_{{\rm Spec}\,k}\Hi^{\prec\Delta}_{S/k}\to\Hi^{\prec\Delta}_{S/k}
  \end{equation*}
  such that 
  \begin{itemize}
    \item the restriction of $g$ to $(t-1)\times_{{\rm Spec}\,k}\Hi^{\prec\Delta}_{S/k}$ is the identity; and 
    \item for each point $p$ of $\Hi^{\prec\Delta}_{S/k}$,
    the restriction of $g$ to $\mathbb{A}^1_{k}\times_{{\rm Spec}\,k}p$ is a curve on $\Hi^{\prec\Delta}_{S/k}$
    which connects $p$ with the point defined by the monomial ideal $(x^\alpha;\alpha\in\mathbb{N}^n-\Delta)\subset S$. 
  \end{itemize}
  In particular, if ${\rm Spec}\,k$ is connected, then $\Hi^{\prec\Delta}_{S/k}$ is connected as well. 
\end{pro}

We omit the proof of this, just noting that it follows the lines of the well-known construction of curves in Hilbert schemes. 
What is needed for the proof is found in the first few pages of \cite{bayer}, in Section 15.8 of \cite{eisenbud}, 
and in Lemma \ref{favouritelemma} above. 
A crucial part is the existence of a linear map $\ell:\mathbb{Z}^n\to\mathbb{Z}$ such that $\ell(\alpha)>\ell(\beta)$
for all $\alpha\in\Delta^{(1)}$ and $\beta\in\Delta$ such that $T_{\alpha,\beta}\neq0$ in the ring $R^{\prec\Delta}_m$. 
This is guaranteed by \cite{bayer}, Chapter 1, \S1, or \cite{eisenbud}, Exercise 15.12. 
Upon writing $\mathbb{A}^1_k$ as ${\rm Spec}\,k[t]$, 
the ring homomorphism $R^{\prec\Delta}\to k[t]\otimes_{k}R^{\prec\Delta}$ corresponding to $g$ 
sends each $T_{\alpha,\beta}$ to $t^{\ell(\alpha)-\ell(\beta)}T_{\alpha,\beta}$. 

Can we carry this construction over from $\Hi^{\prec\Delta}_{S/k}$ to $\Hi^{\Delta}_{S/k}$? 
Here is one obvious case in which we can. 
We say that $\Delta$ a {\it corner cut} if there exists a linear map $\ell:\mathbb{Z}^n\to\mathbb{Z}$ such that
$\ell(\alpha)>\ell(\beta)$ for all $\alpha\in\mathbb{N}^n$ and all $\beta\in\Delta$. 
If $\prec$ is any term order, we define a new term order $\prec_\ell$ 
by first grading the elements of $\mathbb{N}^n$ w.r.t. the weight $\ell$ and then using $\prec$ as a tie-breaker. 
If $\Delta$ is a corner cut and $\ell$ is the corresponding linear map, 
then obviously $\Hi^{\Delta}_{S/k}=\Hi^{\prec_\ell\Delta}_{S/k}$. In particular, $\Hi^{\Delta}_{S/k}$ is connected. 

However, the condition on $\Delta$ being a corner cut is very restrictive. 
Let us replace it by a slightly weaker condition. 
We say that $\Delta$ a {\it weak corner cut} if there exists a linear map $\ell:\mathbb{Z}^n\to\mathbb{Z}$ such that
$\ell(\alpha)\geq\ell(\beta)$ for all $\alpha\in\mathbb{N}^n$ and all $\beta\in\Delta$. 

\begin{ex}
  Remove from the set $\{\alpha\in\mathbb{N}^n;|\alpha|\leq r\}$
  any subset $S$ of $\{\alpha\in\mathbb{N}^n;|\alpha|=r\}$, 
  then the remaining set $\Delta$ is a standard set. 
  $\Delta$ is a weak corner cut but not a corner cut. 
\end{ex}

For showing connectedness if $\Delta$ is a weak corner cut, we need some more notation. 
Consider the functor
\begin{equation*}
  \begin{split}
    \HHi^{\Delta,\ell}_{S/k}:(k\text{-Alg})&\to(\text{Sets})\\
    B&\mapsto
    \left\lbrace
    \begin{array}{c}
      \phi:B[x]\to Q\text{ in $\HHi^{\Delta}_{S/k}(B)$ such that}\\
      \ker\phi\text{ is homogeneous w.r.t. $\ell$}
    \end{array}
    \right\rbrace\,.
  \end{split}
\end{equation*}
Upon using the notation of Proposition \ref{representing} and Theorem \ref{fewer}, 
we see that $\HHi^{\Delta,\ell}_{S/k}$ is representable by the affine subscheme 
$\Hi^{\Delta,\ell}_{S/k}$ of $\Hi^{\Delta}_{S/k}$ defined by the ideal 
$(T_{\alpha,\beta};\alpha\in\Delta\cup\Delta^{(1)},\beta\in\Delta,\ell(\alpha)\neq\ell(\beta))$
in the coordinate ring $R^\Delta$. 

\begin{pro}
  \begin{enumerate}
    \item[(i)] $\Hi^{\Delta,\ell}_{S/k}$ is an affine space. 
    \item[(ii)] If $\Delta$ is a weak corner cut and ${\rm Spec}\,k$ is connected, then $\Hi^{\Delta}_{S/k}$ is connected. 
  \end{enumerate}
\end{pro}

\begin{proof}
  As for (i), we refer to Theorem 5.3 of \cite{krarticle}, where the same statement is proved.
  As for (ii), we start with the same construction as in Proposition \ref{primeconnected}:
  We define a ring homomorphism $R^{\Delta}\to k[t]\otimes_{k}R^{\Delta}$
  by sending each $T_{\alpha,\beta}$ to $t^{\ell(\alpha)-\ell(\beta)}T_{\alpha,\beta}$ 
  and obtain a morphism $g:\mathbb{A}^1_k\times_{{\rm Spec}\,k}\Hi^{\Delta}_{S/k}\to\Hi^{\Delta}_{S/k}$. 
  If $p$ is an arbitrary point in $\Hi^{\prec\Delta}_{S/k}$, 
  then $g((t),p)$ is a point in the closed subscheme $\Hi^{\prec\Delta,\ell}_{S/k}$ of $\Hi^{\prec\Delta}_{S/k}$. 
  Then the assertion follows from (i). 
\end{proof}

The weak corner cut property is not necessary for $\Hi^{\Delta}_{S/k}$ to be connected. 
If $n=2$, or if $d\leq7$, then $\Hi^d_{S/k}$ is known to be irreducible. 
In these cases all open subschemes $\Hi^{\Delta}_{S/k}$, regardless of the shape of $\Delta$, are irreducible, thus connected. 
References for the case $n=2$ include \cite{fogarty}, \cite{millerbernd}, Theorem 18.7 and \cite{hartshorne}, Theorem 8.11. 
The original reference for the case $d\leq7$ is \cite{mazzola}. 
The authors of \cite{velasco} prove that the upper bound $d\leq7$ for irreducibility of $\Hi^d_{S/k}$ is sharp. 
More precisely, Theorem 1.1 of the cited paper states that for $d\leq8$, 
the scheme $\Hi^d_{S/k}$ is reducible if, and only if, $d=8$ and $n\geq4$. 
In that case there exists a component of dimension $8n-7$ in $\Hi^d_{S/k}$
which consists of local algebras isomorphic to homogeneous algebras with Hilbert function $(1,4,3)$. 
Therefore the candidates for $\Delta$ such that $\Hi^{\Delta}_{S/k}$ might be connected are those in $\mathbb{N}^4$
with Hilbert function $(1,4,3)$. These are found in the following list: 
\begin{equation*}
  \begin{split}
    \Delta_1=&\{0,e_1,e_2,e_3,e_4,e_1+e_2,e_1+e_3,e_2+e_3\}\,,\\
    \Delta_2=&\{0,e_1,e_2,e_3,e_4,e_1+e_2,e_1+e_3,2e_1\}\,,\\
    \Delta_3=&\{0,e_1,e_2,e_3,e_4,e_1+e_2,e_1+e_3,2e_2\}\,,\\
    \Delta_4=&\{0,e_1,e_2,e_3,e_4,e_1+e_2,2e_1,2e_2\}\,,\\
    \Delta_5=&\{0,e_1,e_2,e_3,e_4,e_1+e_2,2e_1,2e_3\}\,,\\
    \Delta_6=&\{0,e_1,e_2,e_3,e_4,e_1+e_2,2e_3,2e_4\}\,,\\
    \Delta_7=&\{0,e_1,e_2,e_3,e_4,2e_1,2e_2,2e_3\}\,.
  \end{split}
\end{equation*}
However, it is easy to see that all these sets are weak corner cuts. 
Moreover, it is easy to see that all standard sets of dimension $n\geq5$ and size $d=8$ are weak corner cuts or corner cuts. 
Therefore the question whether or not $\Hi^{\Delta}_{S/k}$ is connected concerns those $\Delta$ of dimension
$n\geq3$ and size $d\geq9$ which are not (weak) corner cuts, for instance 
\begin{equation*}
  \Delta=\{0,e_1,2e_1,3e_1,e_2,2e_2,3e_2,e_3,2e_3\}\subset\mathbb{N}^3\,.
\end{equation*}
For all $\Delta$ in question, 
the presentation of the coordinate ring of $\Hi^{\Delta}_{S/k}$ given by 
Theorem \ref{fewer} is much too large for testing connectedness of $\Hi^{\Delta}_{S/k}$ by computational means. 
It is not known to the author if any of the $\Hi^{\Delta}_{S/k}$ in question is connected or not. 

%%%%%%%%%%%%%%%%%%%%%%%%%%%%%%%%%%%%%%

\section{Changing the charts}\label{changingcharts}

Let $\Delta$ and $\Pi$ be two standard sets of size $d$. 
We embed them into standard sets $\Delta\subset N\subset\mathbb{N}^n$ and $\Pi\subset M\subset\mathbb{N}^n$. 
The union $N\cup M$ is a standard set containing both $\Delta$ and $\Pi$. 
We write the coordinate ring of $\Hi^{\Delta}_{S/k}$ as in Proposition \ref{representing}, 
with $N$ replaced by $N\cup M$:
\begin{equation*}
  R^\Delta=k[T_{\alpha,\beta};\alpha\in N\cup M\cup(N\cup M)^{(1)},\beta\in\Delta]/I^\Delta\,.
\end{equation*}
We write the coordinate ring of $\Hi^{\Pi}_{S/k}$ in an analogous way, 
replacing the matrix $T=(T_{\alpha,\beta})$ of indeterminates by a matrix of new variables $U=(U_{\alpha,\xi})$:
\begin{equation*}
  R^\Pi=k[U_{\alpha,\xi};\alpha\in N\cup M\cup(N\cup M)^{(1)},\xi\in\Pi]/I^\Pi\,.
\end{equation*}
We decompose the indexing set of the rows as follows: 
\begin{equation*}
  N\cup M\cup(N\cup M)^{(1)}=(\Delta\cap\Pi)\coprod(\Delta-\Pi)\coprod(\Pi-\Delta)\coprod\rho\,.
\end{equation*}
Accordingly, we decompose the matrix $T$ into the blocks
\begin{equation*}
  T=\bordermatrix{ 
    & \Delta\cap\Pi & \Delta-\Pi \cr
    \Delta\cap\Pi & E & 0 \cr
    \Delta-\Pi & 0 & E \cr
    \Pi-\Delta & T_{31} & T_{32} \cr
    \rho & T_{41} & T_{42} }\,,
\end{equation*}
and the matrix $U$ into the blocks
\begin{equation*}
  U=\bordermatrix{ 
    & \Delta\cap\Pi & \Pi-\Delta \cr
    \Delta\cap\Pi & E & 0 \cr
    \Delta-\Pi & U_{21} & U_{22} \cr
    \Pi-\Delta & 0 & E \cr
    \rho & U_{41} & U_{42} }\,.
\end{equation*}
where $E$ is the identity matrix. 
The symbols to the left of the rows and above the columns of $T$ and $U$ 
indicate the sets by which the respective submatrices are indexed.  

\begin{pro}\label{prointersection}
  Let $\Delta\subset N$ and $\Pi\subset M$ be standard sets. 
  \begin{enumerate} 
    \item[(i)] $\Hi^{\Delta}_{S/k}\cap\Hi^{\Pi}_{S/k}$ is the open subscheme 
      \begin{equation*}
        {\rm Spec}\,R^\Delta-\mathbb{V}(\det(T_{32}))\,.
      \end{equation*}
    \item[(ii)] The gluing morphism $\psi_{\Delta,\Pi}$ which identifies the intersection 
      as an open subscheme of $\Hi^{\Delta}_{S/k}$ with an open subscheme
      of $\Hi^{\Pi}_{S/k}$ is given by the homomorphism
      \begin{equation*}
        U\mapsto T=U\cdot
        \left(\begin{array}{cc}
        E&0\\
        T_{31}&T_{32}
        \end{array}\right)
      \end{equation*}
      between the coordinate rings.
    \end{enumerate}
\end{pro}

\begin{proof}
  We only prove the first assertion, as the proof of the second is similar. 
  Take a $k$-algebra $B$ and a homomorphism which lies in both $\HHi^{\Delta}_{S/k}(B)$ 
  and $\HHi^{\Pi}_{S/k}(B)$. 
  This homomorphism is represented by two surjections $\phi$ and $\phi^\prime$, respectively,
  such that there exists an isomorphism $\Psi$ making the following diagram commutative:
  \begin{equation*}
    \begin{CD}
      Bx^\Pi@>>>B[x]@>\phi^\prime>>Bx^\Pi\\
      @V\Psi V\cong V @V{\rm id}VV @V\Psi V\cong V\\
      Bx^\Delta@>>>B[x]@>\phi>>Bx^\Delta\,.
    \end{CD}
  \end{equation*}
  For all $\alpha\in\Pi-\Delta$, consider the elements $f_{\alpha}\in\ker\phi$ 
  of equation \eqref{border}. 
  From the commutative diagram above it follows that 
  \begin{equation}\label{psi}
    \Psi(x^\alpha)=
    \begin{cases}
      \,\,\,x^\alpha&\text{ if }\alpha\in\Pi\cap\Delta\,,\\
      -\sum_{\beta\in\Delta}d_{\alpha,\beta}x^\beta&\text{ if }\alpha\in\Pi-\Delta\,.
    \end{cases}
  \end{equation}
  Indeed, the first line is immediate; as for the second line, if $\alpha\in\Pi-\Delta$, 
  then $\phi(x^\alpha+\sum_{\beta\in\Delta}d_{\alpha,\beta}x^\beta)=0$, i.e.
  \begin{equation*}
    \Psi(x^\alpha)=\Psi(\phi^\prime(x^\alpha))=\phi(x^\alpha)
    =-\sum_{\beta\in\Delta}d_{\alpha,\beta}\phi(x^\beta)
    =-\sum_{\beta\in\Delta}d_{\alpha,\beta}x^\beta\,.
  \end{equation*}
  As for the inverse of $\Psi$, we define the polynomials 
  \begin{equation*}
    g_{\alpha}=x^\alpha+\sum_{\beta\in\Pi}e_{\alpha,\beta}x^\beta\in\ker\phi^\prime
  \end{equation*}
  in analogy to \eqref{border} and obtain
  \begin{equation*}
    \Psi^{-1}(x^\alpha)=
    \begin{cases}
      \,\,\,x^\alpha&\text{ if }\alpha\in\Pi\cap\Delta\,,\\
      -\sum_{\beta\in\Delta}e_{\alpha,\beta}x^\beta&\text{ if }\alpha\in\Delta-\Pi\,.
    \end{cases}
  \end{equation*}
  The given homomorphism lies in both $\Hi^{\Delta}_{S/k}$ and $\Hi^{\Pi}_{S/k}$ if, and only if, 
  the linear map \eqref{psi} is invertible. 
  We see that this condition is equivalent to the matrix 
  \begin{equation*}
    (d_{\alpha,\beta})_{\alpha\in\Pi-\Delta,\,\beta\in\Delta-\Pi}
  \end{equation*}
  being invertible. 
\end{proof}

Note that this implies that $\Hi^{\Delta}_{S/k}\cap\Hi^{\Pi}_{S/k}$ is the open locus where the matrices
\begin{equation*}
  T^{\Box}=\left(\begin{array}{cc}
  E&0\\
  T_{31}&T_{32}
  \end{array}\right)\,\text{ and }\,,
  U^{\Box}=\left(\begin{array}{cc}
  E&0\\
  U_{21}&U_{22}
  \end{array}\right)
\end{equation*}
are inverse to each other. 
This is also proved in Section 2.2 of \cite{huibregtse}. 

Consider the matrix of indeterminates $T=(T_{\alpha,\beta})$ from the above discussion. 
The rows and columns of that matrix are indexed by elements of $\mathbb{N}^n$, 
which are ordered by the term order $\prec$. 
By Corollary \ref{prime}, $\Hi^{\prec\Delta}_{S/k}$ is the closed subscheme of ${\rm Spec}\,R^\Delta$
on which $T$ is a {\it lower triangular matrix w.r.t. $\prec$}. 
We obtain:

\begin{cor}
  The intersection $\Hi^{\prec\Delta}_{S/k}\cap\Hi^{\Pi}_{S/k}$ is the 
  locally closed subscheme of $\Hi^{\Pi}_{S/k}$ in which 
  \begin{itemize}
    \item $U^{\Box}$ is the inverse of $T^{\Box}$, and
    \item $T$ is a lower triangular matrix w.r.t. $\prec$. 
  \end{itemize}
\end{cor}

It is thus tempting to suspect that the boundary in $\Hi^{\Pi}_{S/k}$ of 
$\Hi^{\prec\Delta}_{S/k}\cap\Hi^{\Pi}_{S/k}$ is just 
$\Hi^{\prec\Pi}_{S/k}$. 
Indeed, ``the inverse of a lower triangular matrix is a lower triangular matrix'', 
hence from $T^\Box U^\Box=E$, we might deduce that $U^\Box$ is lower triangular. 
Moreover, ``the product of two lower triangular matrices is lower triangular'', 
hence $U$, which by Proposition \ref{prointersection} can be written as $U=TU^\Box$, must be lower triangular. 

However, the first quoted assertion only holds for square matrices in which rows and columns are
indexed by the same totally ordered set. 
This is not the case for $T^\Box$ and $U^\Box$, whose rows and columns are indexed by $\Delta$ and $\Pi$, resp. 
(The second quoted assertion is true for all products $AB$ of matrices $A$ and $B$ 
indexed by subsets of a totally ordered set such that the indexing set of the columns of $A$ 
equals the indexing set of the rows of $B$.)
Moreover, if $T$ and $U$ were both lower triangular in a point $\mathfrak{p}\in{\rm Spec}\,k$, 
that point would lie in both $\Hi^{\prec\Delta}_{S/k}$ and $\Hi^{\prec\Pi}_{S/k}$, 
in contradiction to Theorem \ref{coprod}. 
The most we can hope for is that 
\begin{equation}\label{partial}
  \partial(\Hi^{\prec\Delta}_{S/k}\cap\Hi^{\Pi}_{S/k})=\Hi^{\prec\Pi}_{S/k}\,, 
\end{equation}
as indicated above. 
If this was true for all $\Delta$ and $\Pi$, 
the decomposition of Theorem \ref{coprod} would be a {\it stratification}. 
Indeed, denote by $\mathcal{S}$ the set of all standard sets of size $d$.
Then the decomposition of Theorem \ref{coprod} is a stratification if, and only if, 
for all $\Delta\in\mathcal{S}$, there exists a subset $\mathcal{S}(\Delta)\subset\mathcal{S}$ such that
\begin{equation}\label{closure}
  \overline{\Hi^{\prec\Delta}_{S/k}}=\coprod_{\Pi\in\mathcal{S}(\Delta)}\Hi^{\prec\Pi}_{S/k}\,.
\end{equation}
If this is true, then 
\begin{equation*}
  \begin{split}
    \mathcal{S}(\Delta)&=\{\Pi\in\mathcal{S};
    \Hi^{\prec\Delta}_{S/k}\cap\Hi^{\Pi}_{S/k}\neq\emptyset\}\\
    &=\{\Pi\in\mathcal{S};
    \overline{\Hi^{\prec\Delta}_{S/k}}\cap\Hi^{\Pi}_{S/k}\neq\emptyset\}\\
    &=\{\Pi\in\mathcal{S};
    \overline{\Hi^{\prec\Delta}_{S/k}}\cap\Hi^{\prec\Pi}_{S/k}\neq\emptyset\}\,.
  \end{split}
\end{equation*}
Therefore it is clear that \eqref{partial} holds for all $\Delta,\Pi\in\mathcal{S}$ if, and only if, 
for all $\Delta\in\mathcal{S}$ there exists an $\mathcal{S}(\Delta)\subset\mathcal{S}$ 
such that \eqref{closure} holds. 
Here is a negative result on the question wether or not we have a stratification here. 

\begin{pro}
  If $\prec$ is the lexicographic order on $S$ and $\dim\Hi^{d}_{S/k}>dn$, 
  then the decomposition of Theorem \ref{coprod} is not a stratification. 
  In particular, if $n=3$ and $d\geq102$ or $n=4$ and $d\geq25$, 
  the decomposition of Theorem \ref{coprod} is not a stratification. 
\end{pro}

\begin{proof}
  We order the variables such that $x_{1}\succ\ldots\succ x_{n}$ and consider the standard set
  $\Delta=\{0,\ldots,(r-1)e_{n}\}\subset\mathbb{N}^n$. 
  Remember that by Example \ref{axis}, $\Hi^\Delta_{S/k}$ has dimension $dn$. 
  Furthermore, $\Hi^\Delta_{S/k}=\Hi^{\prec\Delta}_{S/k}$, 
  as $\alpha\succ\beta$ for all $\alpha\in\mathbb{N}^n-\Delta$ and all $\beta\in\Delta$ in the lexicographic order. 
  We claim that 
  \begin{equation}\label{max}
    \forall\Pi\in\mathcal{S}:\overline{\Hi^{\prec\Delta}_{S/k}}\cap\Hi^{\prec\Pi}_{S/k}\neq\emptyset\,.
  \end{equation}
  For proving that, we use a few schemes introduced and discussed in Section 5 of \cite{components}. 
  The first is
  \begin{equation*}
    \Hi^{d,0}_{S/k}=((\mathbb{A}^n_{k})^d-\Lambda)/S_{r}\,,
  \end{equation*}
  where for $i=1,\ldots,r$, we denote by $(x^{(i)}_{1},\ldots,x^{(i)}_{n})$
  the coordinates on the $i$-th copy of $\mathbb{A}^n$ in the product;
  where $\Lambda=\cup_{i\neq j}\mathbb{V}(x^{(i)}_{1}-x^{(j)}_{1},\ldots,x^{(i)}_{n}-x^{(j)}_{n})$
  is the {\it large diagonal} in the product; 
  and where $S_{r}$ is the symmetric group acting on the product in the obvious way. 
  $\Hi^{d,0}_{S/k}$ is an open subscheme of $\Hi^{d}_{S/k}$, 
  and the functor associated to the scheme $\Hi^{d,0}_{S/k}$ is given by 
  \begin{equation*}
    \HHi^{d,0}_{S/k}(B)=  
    \left\lbrace
    \begin{array}{c}
      \text{ closed subschemes }Z\subset\mathbb{A}^n_{B}\text{ such that }\\
      p:Z\to{\rm Spec}\,B\text{ is finite \'{e}tale of degree $d$}
    \end{array}
  \right\rbrace\,,
  \end{equation*}
  where $p$ is the restriction to $Z$ of the projection $\mathbb{A}^n_{B}\to{\rm Spec}\,B$. 
  (In contrast to that, the functor $\HHi^{d}_{S/k}$ sends a $k$-algebra $B$
  to the set of all closed subschemes $Z\subset\mathbb{A}^n_{B}$ 
  such that the restriction of the projection $p=\mathbb{A}^n_{B}\to{\rm Spec}\,B$ is finite flat of degree $d$.
  Finite flatness of $p:Z={\rm Spec}\,B[x]/I\to{\rm Spec}\,B$ translates to local freeness of the $B$-algebra $B[x]/I$.
  Therefore, the additional requirement in $\HHi^{d,0}_{S/k}$ is {\it unramifiedness}.)
  Moreover, we use the subscheme
  \begin{equation*}
    \Hi^{\prec\Delta,0}_{S/k}=\Hi^{\Delta}_{S/k}\cap\Hi^{d,0}_{S/k}
  \end{equation*}
  of $\Hi^{d,0}_{S/k}$ and its analogue for $\Pi$ instead of $\Delta$. 
  The functor associated to that scheme sends a $k$-algebra $B$ 
  to the set of all \'{e}tale $p:Z={\rm Spec}\,B[x]/I\to{\rm Spec}\,B$
  such that $I\subset B[x]$ is monic with standard set $\Delta$. 
  (As $\Hi^{\Delta}_{S/k}=\Hi^{\prec\Delta}_{S/k}$, the functor equivalently sends a $k$-algebra $B$ 
  to the set of all \'{e}tale $p:Z={\rm Spec}\,B[x]/I\to{\rm Spec}\,B$ such that $B[x]/I$ is free with basis $x^\Delta$.) 
  
  We fix a homomorphism from $k$ to a field $k^\prime$ having at least $d$ elements
  and a bijection $\{0,\ldots,d-1\}\to B$, where $B\subset k^\prime$ has $d$ elements. 
  That induces a bijection $\{0,\ldots,d-1\}^n\to B^n$, whose restriction to $\Pi$ induces a bijection 
  $\Pi\to C$, where $C\subset B^n\subset(k^\prime)^n$ has $d$ elements. 
  Let $I\subset k^\prime[x]$ be the ideal defining $C$. 
  Then Corollary 10 of \cite{jpaa} says that the ideal $I$ is monic with standard set $\Pi$. 
  Therefore $\xi=k^\prime[x]/I$ is a $k^\prime$-rational closed point of $\Hi^{\prec\Pi}_{S/k}$. 
  We shall prove that $\xi$ lies in the closure in $\Hi^{d,0}_{S/k}$ of $\Hi^{\prec\Delta,0}_{S/k}$. 
  Then $\xi$ will also lie in the closure in $\Hi^{d}_{S/k}$ of $\Hi^{\prec\Delta}_{S/k}$, 
  and \eqref{max} will be proved. 
  
  For this we denote by
  \begin{equation*}
    \pi:(\mathbb{A}^n_{k})^d-\Lambda\to\Hi^{d,0}_{S/k}
  \end{equation*}
  the canonical morphism. As $\Hi^{\prec\Delta}_{S/k}=\Hi^{\Delta}_{S/k}$, 
  and therefore $\Hi^{\prec\Delta,0}_{S/k}$ is an open subscheme of $\Hi^{d,0}_{S/k}$, 
  it follows that $\pi^{-1}(\Hi^{\Delta,0}_{S/k})=(\mathbb{A}^n_{k})^d-\Lambda-A$
  for some closed $A\subset(\mathbb{A}^n_{k})^d$. 
  (In fact, $A$ is the scheme associated to the ideal $(x_{n}^{i}-x_{n}^{j};i\neq j)$, but we will not need that here.)
  Remember that we have to show that for all open $U\subset\Hi^{d,0}_{S/k}$ such that $\xi\in U$, 
  we have $U\cap\Hi^{\prec\Delta,0}_{S/k}\neq\emptyset$. 
  Let $\eta$ be an element of $\pi^{-1}(\xi)$. 
  (The choice of $\eta$ corresponds to the choice of a labeling of the elements of $E$.)
  It clearly suffices to show that for all open $V\subset(\mathbb{A}^n_{k})^d-\Lambda$ 
  such that $\eta\in V$, we have $V\cap((\mathbb{A}^n_{k})^d-\Lambda-A)\neq\emptyset$. 
  Upon writing $V=(\mathbb{A}^n_{k})^d-\Lambda-A^\prime$, 
  for some closed $A^\prime\subset(\mathbb{A}^n_{k})^d$, 
  it follows that $V\cap((\mathbb{A}^n_{k})^d-\Lambda-A)=(\mathbb{A}^n_{k})^d-\Lambda-A-A^\prime$. 
  That intersection is empty if, and only if, $\Lambda\cup A\cup A^\prime=(\mathbb{A}^n_{k})^d$. 
  But this is impossible as all three summands are closed subschemes of $(\mathbb{A}^n_{k})^d$
  and strictly smaller than the ambient scheme. 
  
  Now that \eqref{max} is proved, we conclude as follows:
  If the decomposition in question was a stratification, \eqref{max} would say that
  $\mathcal{S}(\Delta)=\mathcal{S}$. 
  Moreover, each $\Hi^{\prec\Pi}_{S/k}$, being a subscheme of the closure of $\Hi^{\prec\Delta}_{S/k}$, 
  would have a dimension at most $dn$. 
  Therefore by Theorem \ref{coprod} and \eqref{closure}, the dimension of 
  \begin{equation*}
    \Hi^d_{S/k}=\coprod_{\Pi\in\mathcal{S}}\Hi^{\prec\Pi}_{S/k}
    =\coprod_{\Pi\in\mathcal{S}(\Delta)}\Hi^{\prec\Pi}_{S/k}
    =\overline{\Hi^{\prec\Delta}_{S/k}}\,.
  \end{equation*}
  would be $dn$. 
  This proves the first assertion of the proposition. 
  As for the second assertion, we know from (1) of \cite{iarrobino} that $\dim\Hi^d_{S/k}>dn$
  if $n=3$ and $d\geq102$ or $n=4$ and $d\geq25$. 
\end{proof}

The last proposition raises some interesting questions. 
The {\it good component} $\mathscr{G}^d_{S/k}$ of $\Hi^{d}_{S/k}$ is the schematic closure of the subscheme $\Hi^{d,0}_{S/k}$. 
(As for the construction of $\mathscr{G}^d_{S/k}$, see \cite{ekedahlskjelnes} or \cite{rydhskjelnes}.)
In particular, $\mathscr{G}^d_{S/k}$ is of dimension $dn$, thus the argument we used in the last proof does not work 
if we replace the full Hilbert scheme by the good component. 

\begin{q}
  Is the decomposition 
  \begin{equation*}
    \mathscr{G}^d_{S/k}=\coprod_{\Delta}(\mathscr{G}^d_{S/k}\cap\Hi^{\prec\Delta}_{S/k})\,,
  \end{equation*}
  induced by the decomposition of Theorem \ref{coprod}, a stratification?
\end{q}

%%%%%%%%%%%%%%%%%%%%%%%%%%%%%%%%%%%%%%

\section{Acknowledgements}

I wish to thank my colleagues Michael Spie{\ss}, Thomas Zink, Eike Lau and Vytautas Paskunas 
at the University of Bielefeld for being there for me and my questions at any time. 
I particularly thank Mike Stillman for his great support and warm welcome when I moved from Bielefeld to Cornell University, 
for showing great interest in the equations of Theorem \ref{fewer}, 
and for experimenting with these equations in {Macaulay2} \cite{M2}.
Many thanks go to Elmar Gro{\ss}e Kl\"onne at Humboldt Universt\"at Berlin 
for giving me the opportunity to talk about this piece of work while it was still under construction, 
and to Roy Skjelnes, Lorenzo Robbiano and Paolo Lella 
for a number of valuable suggestions on a first draft of the present paper.

\bibliography{strata.bib}
\bibliographystyle{amsalpha}

\end{document}